\documentclass[oneside,final,11pt]{amsart}

\usepackage{dsfont}
\usepackage{amssymb}

\newcommand{\R}{\mathds R}
\newcommand{\dd}{\mathrm d}
\newcommand{\chilow}[1]{\chi_{\lower2pt\hbox{$\scriptstyle#1$}}}

\DeclareMathOperator{\Der}{Der}
\DeclareMathOperator{\DA}{DA}
\DeclareMathOperator{\BV}{BV}
\DeclareMathOperator{\NBV}{NBV}

\title[Extension of $c_0(I)$-valued operators]{Extension of $\mathbf{c_0(I)}$-valued operators on spaces of continuous functions on compact lines}

\author{Victor dos Santos Ronchim}
\address{(V. S. Ronchim) Departamento de Matem\'atica,\hfill\break\indent Universidade de S\~ao Paulo, Brazil}
\email{vronchim@ime.usp.br}
\author{Daniel V. Tausk}
\address{(D. V. Tausk) Departamento de Matem\'atica,\hfill\break\indent Universidade de S\~ao Paulo, Brazil}
\email{tausk@ime.usp.br}
\urladdr{http://www.ime.usp.br/\~{}tausk}
\subjclass[2020]{46B26,46E15,54F05}
\keywords{Banach spaces of continuous functions; extensions of bounded operators; compact lines}

\date{May 28th, 2022}

\begin{document}

\theoremstyle{plain}\newtheorem{teo}{Theorem}[section]
\theoremstyle{plain}\newtheorem{prop}[teo]{Proposition}
\theoremstyle{plain}\newtheorem{lem}[teo]{Lemma}
\theoremstyle{plain}\newtheorem{cor}[teo]{Corollary}
\theoremstyle{definition}\newtheorem{defin}[teo]{Definition}
\theoremstyle{remark}\newtheorem{rem}[teo]{Remark}
\theoremstyle{plain} \newtheorem{assum}[teo]{Assumption}
\theoremstyle{definition}\newtheorem{example}[teo]{Example}

\begin{abstract}
We investigate the problem of existence of a bounded extension to $C(K)$ of a bounded $c_0(I)$-valued operator $T$ defined on the subalgebra of $C(K)$
induced by a continuous increasing surjection $\phi:K\to L$, where $K$ and $L$ are compact lines. Generalizations of some of the results of \cite{CTlines} about extension of $c_0$-valued operators are obtained. For instance, we prove that when a bounded extension of $T$ exists then an extension can be obtained with norm at most twice the norm of $T$. Moreover, the class of compact lines $L$ for which the $c_0$-extension property is equivalent to the $c_0(I)$-extension property for any continuous increasing surjection $\phi:K\to L$ is studied.
\end{abstract}

\maketitle

\begin{section}{Introduction}

The problem of extension of bounded linear operators is a classical problem in the theory of Banach spaces. It includes, as a particular case, the problem of complementation
of closed subspaces since a closed subspace $Y$ of a Banach space $X$ is complemented in $X$ if and only if the identity operator of $Y$ admits a bounded extension to $X$. The Banach spaces $Z$ for which every bounded operator $T:Y\to Z$ admits a bounded extension to $X$ (for all $X$ and $Y$) are known as {\em injective\/} Banach spaces. For example, the space $\ell_\infty$ is {\em isometrically injective}, i.e., every bounded linear operator $T:Y\to\ell_\infty$ admits an extension to $X$ having the same norm as $T$. More generally, the space $\ell_\infty(I)$ of bounded families of real numbers indexed by an arbitrary set $I$ (endowed with the sup norm) is isometrically injective. This is a simple consequence of the Hahn--Banach Theorem as bounded operators $T:X\to\ell_\infty(I)$ are in one-to-one correspondence with bounded families $(\alpha_i)_{i\in I}$ in $X^*$, where $\alpha_i$ is the $i$-th coordinate functional of $T$; in addition, $\Vert T\Vert=\sup_{i\in I}\Vert\alpha_i\Vert$.

The subspace $c_0$ of $\ell_\infty$ consisting of sequences that converge to zero is not injective as the identity of $c_0$ does not extend \cite{Phillips} to $\ell_\infty$. However, the classical Theorem of Sobczyk \cite{Sobczyk} states that $c_0$ is {\em separably injective}, i.e., every bounded operator $T:Y\to c_0$ admits a bounded extension to $X$ for any closed subspace $Y$ of a separable Banach space $X$. Moreover, the extension can be chosen with norm bounded by twice the norm of $T$. It is easy to see that the statement of Sobczyk's Theorem remains valid when $c_0$ is replaced by its nonseparable version $c_0(I)$. Recall that for an arbitrary set $I$, one defines the space $c_0(I)$ to be the subspace of $\ell_\infty(I)$ consisting of families $(a_i)_{i\in I}$ such that $\lim_{i\in I}a_i=0$, i.e., such that the set $\big\{i\in I:\vert a_i\vert\ge\varepsilon\big\}$ is finite for every $\varepsilon>0$.

Generalizations of Sobczyk's Theorem is a widely studied topic \cite{ArgyrosLondon,Caudia,CTJFA,CTResults,CT,CTlines,DrPlebanek,Eloi,Molto,Patterson}. In \cite{CT} the authors introduced the following definition: a closed subspace $Y$ of a Banach space $X$ is said to have the {\em $c_0$-extension property\/} in $X$ if every bounded operator $T:Y\to c_0$ admits a bounded extension to $X$. Sobczyk's Theorem then states that every closed subspace $Y$ of a separable Banach space $X$ has the $c_0$-extension property in $X$ and one can investigate generalizations of Sobczyk's Theorem by studying conditions under which $Y$ has the $c_0$-extension property in $X$ for a nonseparable Banach space $X$. Clearly, bounded operators $T:X\to c_0$ are in one-to-one correspondence with sequences $(\alpha_n)_{n\ge1}$ in $X^*$ that are {\em weak*-null}, i.e., sequences that converge to zero in the weak*-topology. Thus the problem of whether a bounded operator $T:Y\to c_0$ admits a bounded extension to $X$ is equivalent to the problem of whether a weak*-null sequence in $Y^*$ admits a (term by term) extension to a weak*-null sequence in $X^*$.

Another natural direction to look for generalizations of Sobczyk's Theorem is to replace $c_0$ by $c_0(I)$ and the notion of $c_0$-extension property by the notion of $c_0(I)$-extension property. For instance, in \cite{ArgyrosLondon} the authors studied the problem of complementation of isomorphic copies of the space $c_0(I)$ in a Banach space $X$. This is a particular case of the more general problem of studying closed subspaces of $X$ with the $c_0(I)$-extension property in $X$, since an isomorphic copy $Y$ of $c_0(I)$ in $X$ has the $c_0(I)$-extension property in $X$ if and only if it is complemented in $X$. We note that bounded operators $T:X\to c_0(I)$ are in one-to-one correspondence with families $(\alpha_i)_{i\in I}$ in $X^*$ that are {\em weak*-null\/} in the sense that $\lim_{i\in I}\alpha_i(x)=0$ for all $x\in X$.

The problem of characterizing pairs $(X,Y)$ such that $Y$ has the $c_0$-ex\-tension property or the $c_0(I)$-extension property in $X$ is too general to be handled and thus it is necessary to work with particular classes of Banach spaces $X$ and closed subspaces $Y$. An important class in which this problem has been studied consists of spaces $C(K)$ of real-valued continuous functions on a compact Hausdorff space $K$ endowed with the sup norm. Since $C(K)$ is a Banach algebra with unity, it is natural to consider subspaces $Y$ of $C(K)$ that are Banach subalgebras with unity. Note that it follows from Stone--Weierstrass Theorem that such subalgebras are precisely the images of composition operators $\phi^*:C(L)\to C(K)$, where $L$ is a compact Hausdorff space and $\phi:K\to L$ is a surjective continuous map. As usual, the {\em composition operator\/} $\phi^*$ associated to a continuous map $\phi$ is the operator defined by $\phi^*(f)=f\circ\phi$, for all $f$. In this context, we say that $\phi:K\to L$ has the {\em $c_0$-extension property\/} (resp., {\em $c_0(I)$-extension property}) if the image of $\phi^*$ has the $c_0$-extension property (resp., $c_0(I)$-extension property) in $C(K)$.

In \cite{CTlines} the authors obtained a characterization of the continuous surjections $\phi:K\to L$ with the $c_0$-extension property when the spaces $K$ and $L$ are compact lines and $\phi$ is increasing. By a {\em compact line\/} we mean a totally ordered space which is compact in the order topology. The case where $\phi$ is not increasing is much harder and the problem was solved in \cite{CT} only when $L$ is countable. The main goal of the present paper is to generalize some of the results of \cite{CTlines} to the case when the $c_0$-extension property is replaced with the $c_0(I)$-extension property.

In what follows we give a bird's eye view of the contents of the paper and we introduce some of the terminology that will be used later on. Recall that if $K$ is a compact Hausdorff space then the dual space of $C(K)$ is isometrically identified with the space $M(K)$ of finite regular countably-additive signed Borel measures on $K$ endowed with the total variation norm. When $K$ is a compact line, a more useful representation of bounded linear functionals on $C(K)$ can be obtained as $M(K)$ can be identified with the space $\NBV(K)$ of right-continuous maps $F:K\to\R$ of bounded variation (see Proposition~\ref{thm:propNBVK} for details). Bounded operators $T:C(K)\to c_0(I)$ are then in one-to-one correspondence with families $(F_i)_{i\in I}$ in $\NBV(K)$ that are {\em weak*-null\/} in the sense that the corresponding family in $C(K)^*$ is weak*-null. Given a continuous increasing surjection $\phi:K\to L$ between compact lines $K$ and $L$ one can then ask for a characterization of the weak*-null families $(F_i)_{i\in I}$ in $\NBV(L)$ for which the corresponding bounded operator $T:C(L)\to c_0(I)$ admits a bounded extension to $C(K)$; here we identify $C(L)$ with its image under the isometric embedding $\phi^*:C(L)\to C(K)$, so that by a {\em bounded extension\/} of $T:C(L)\to c_0(I)$ we mean a bounded operator $S:C(K)\to c_0(I)$ with $S\circ\phi^*=T$. When the bounded operator $T:C(L)\to c_0(I)$ associated to a weak*-null family $(F_i)_{i\in I}$ in $\NBV(L)$ admits a bounded extension to $C(K)$ we will say that the family $(F_i)_{i\in I}$ {\em extends through $\phi$}.

In \cite{CTlines} the authors present a simple characterization of weak*-null families $(F_i)_{i\in I}$ that extend through a continuous increasing surjection $\phi:K\to L$ in the case where $I$ is countable. In this characterization, a central role is played by the subset $Q(\phi)$ of $L$ consisting of points with nontrivial preimage:
\[Q(\phi)=\big\{t\in L:\text{$\phi^{-1}(t)$ has more than one point}\big\}.\]
We observe that the set $Q(\phi)$ also appeared in \cite[Lemma~2.7]{KK} in a characterization of continuous increasing surjections $\phi:K\to L$ such that
the image of $\phi^*$ is complemented in $C(K)$.
The characterization appearing in \cite[Proposition~2.5]{CTlines} is the following: a weak*-null sequence $(F_n)_{n\ge1}$ in $\NBV(L)$ extends through $\phi:K\to L$ if and only if the set
\begin{equation}\label{eq:Fnnaotendezero}
\big\{t\in Q(\phi):\big(F_n(t)\big)_{n\ge1}\not\in c_0\big\}
\end{equation}
is countable. Using this result, a characterization of continuous increasing surjections $\phi$ with the $c_0$-extension property was obtained in \cite[Theorem~2.6]{CTlines}.

Characterizing weak*-null families $(F_i)_{i\in I}$ that extend through a continuous increasing surjection $\phi:K\to L$ when $I$ is uncountable is much more challenging and a simple adaptation of \cite[Proposition~2.5]{CTlines} does not seem to be possible. A first step towards the solution of this problem is to translate extendibility of $(F_i)_{i\in I}$ through $\phi$ into a condition that involves the map $\phi$ only through the set $Q(\phi)$. This translation is attained using the notion of a family
of real-valued maps of type $c_0\ell_1$ over a set that is introduced and studied in Section~\ref{sec:c0ell1}. More precisely, we prove that a weak*-null family
$(F_i)_{i\in I}$ in $\NBV(L)$ extends through $\phi:K\to L$ if and only if $(F_i)_{i\in I}$ is of type $c_0\ell_1$ over the set $Q(\phi)$. We then proceed to study necessary and sufficient conditions for a family of real-valued maps $(F_i)_{i\in I}$ to be of type $c_0\ell_1$ over a set $Q$ and how this property is related to the set
\begin{equation}\label{eq:QFinaotendezero}
\big\{t\in Q:\big(F_i(t)\big)_{i\in I}\not\in c_0(I)\big\},
\end{equation}
which is the analogue of \eqref{eq:Fnnaotendezero} in this context. This theory is used for instance to establish that if a bounded operator $T:C(L)\to c_0(I)$ admits a bounded extension to $C(K)$ then an extension with norm bounded by $2\Vert T\Vert$ exists. This was already proven for countable $I$ in \cite[Proposition~2.5]{CTlines} using Sobczyk's Theorem, but the general case considered in this paper is much harder and require new techniques. It is interesting to observe that the extension constant $2$ often appears in the theory of extension of $c_0$-valued operators (see for instance \cite{Caudia}), though \cite[Theorem~1.1]{ArgyrosLondon} suggests that for $c_0(I)$-valued operators the constant might in general depend on the cardinality of the set $I$. The theory studied in Section~\ref{sec:c0ell1} also leads to a nice characterization of families $(F_i)_{i\in I}$ in $\NBV(K)$ that are weak*-null and this is presented in Subsection~\ref{sub:weakstarnull}. The problem of characterizing weak*-null families is interesting on its own and it was studied in \cite{Hognas} in the particular case when $K=[0,1]$.

In Section~\ref{sec:separablydetermined} we study compact lines $L$ for which one can check that a weak*-null family $(F_i)_{i\in I}$ in $\NBV(L)$ is of type $c_0\ell_1$ over a subset $Q$ of $L$ by only looking at the intersection of \eqref{eq:QFinaotendezero} with closed separable subsets of $L$ (see Definition~\ref{thm:defseparablydetermined} for details); we call such compact lines {\em separably determined}. It turns out that a compact line $L$ is separably determined if and only if the $c_0$-extension property is equivalent to the $c_0(I)$-extension property for every continuous increasing surjection $\phi:K\to L$ and every compact line $K$. It is shown
that the class of separably determined compact lines is large and closed under important constructions, such as finite products endowed with the lexicographic order. However, in Section~\ref{sec:notsepdet} we prove that the $\omega$-th power of an uncountable compact line, ordered lexicographically, is not separably determined. The proof of this result leads naturally to the topological notion of countably separating families (Definition~\ref{thm:defcountsep}) which is shown to have a deeper relation with families of type $c_0\ell_1$ in Subsection~\ref{sub:countsepc0ell1}. This result also provides an example of a closed subspace of a $C(K)$ space which has the $c_0$-extension property but not the $c_0(I)$-extension property with $I$ of cardinality $\omega_1$. An example of this type was already known \cite[Theorem~1.1~(b)]{ArgyrosLondon}, but for $I$ of much larger cardinality (see Remark~\ref{thm:remsmallercardinal}).

\end{section}

\begin{section}{Operator extension and families of type $c_0\ell_1$}
\label{sec:c0ell1}

We start by introducing some terminology and by recalling basic facts about compact lines. We call a point $t$ of a compact line $K$ {\em right isolated\/} (resp., {\em left isolated}) if either $t=\max K$ (resp., $t=\min K$) or there exists
$s\in K$ with $s>t$ and $\left]t,s\right[=\emptyset$ (resp., with $s<t$ and $\left]s,t\right[=\emptyset$). The set of right isolated (resp., left isolated) points of
$K$ is denoted by $K^+$ (resp., $K^-$) and the points of $K$ not in $K^+\cup K^-$ are called {\em internal}.

Given an arbitrary map $F:K\to\R$ defined on a compact line $K$, we set $V(F;P)=\sum_{i=0}^{n-1}\vert F(t_{i+1})-F(t_i)\vert$ for a finite subset
$P=\{t_0,t_1,\ldots,t_n\}$ of $K$ with $t_0<t_1<\cdots<t_n$. The {\em total variation\/} $V(F)\in[0,+\infty]$ of $F$ is defined as the supremum of $V(F;P)$ with $P$ ranging
over all finite subsets of $K$ and we set $\Vert F\Vert_{\BV}=\vert F(\min K)\vert+V(F)$. The map $F$ is said to be of {\em bounded variation\/} if $V(F)<+\infty$. The space
$\BV(K)$ of all maps $F:K\to\R$ of bounded variation is a Banach space endowed with the norm $\Vert\cdot\Vert_{\BV}$. A map $F:K\to\R$ of bounded variation
admits a finite right-limit $\lim_{s\to t^+}F(s)$ (resp., a finite left-limit $\lim_{s\to t^-}F(s)$) at every point $t\in K$ that is not right-isolated (resp., not left-isolated). We denote by $\NBV(K)$ the closed subspace of $\BV(K)$ consisting of maps that are right-continuous. The following representation theorem for $M(K)\equiv C(K)^*$
will be used throughout the article.
\begin{prop}\label{thm:propNBVK}
Let $K$ be a compact line. For every $\mu\in M(K)$, the map $F_\mu:K\to\R$ defined by $F_\mu(t)=\mu\big([\min K,t]\big)$ for all $t\in K$
is in $\NBV(K)$ and the operator $M(K)\ni\mu\mapsto F_\mu\in\NBV(K)$ is a linear isometry.
\end{prop}
\begin{proof}
See \cite[Lemma~3.1]{CTJFA}.
\end{proof}

Another basic fact that will be used without further mention is that if $t\in K^+$ is a right-isolated point of a compact line $K$ then $\lim_{i\in I}F_i(t)=0$ for every weak*-null family $(F_i)_{i\in I}$ in $\NBV(K)$. This follows from the observation that the interval $[\min K,t]$ is clopen and therefore its characteristic function $\chilow{[\min K,t]}$ is in $C(K)$.

Given a continuous map $\phi:K\to L$, where $K$ and $L$ are arbitrary compact Hausdorff spaces, it is easily seen that the adjoint of the composition operator $\phi^*:C(L)\to C(K)$ is identified with the push-forward operator $\phi_*:M(K)\to M(L)$ defined by $\phi_*(\mu)(B)=\mu\big(\phi^{-1}[B]\big)$, for all $\mu\in M(K)$ and every Borel subset $B$ of $L$. If $K$ and $L$ are compact lines, the isometry given in Proposition~\ref{thm:propNBVK} yields an identification of $\phi_*:M(K)\to M(L)$ with an operator from $\NBV(K)$ to $\NBV(L)$ (also denoted by $\phi_*$) and if $\phi$ is increasing and surjective such operator $\phi_*:\NBV(K)\to\NBV(L)$ is given by
\begin{equation}\label{eq:defphistarF}
\phi_*(G)(t)=G\big(\!\max\phi^{-1}(t)\big),
\end{equation}
for all $t\in L$ and every $G\in\NBV(K)$. We will also adopt equality \eqref{eq:defphistarF} as the definition of $\phi_*(G)$ for an arbitrary map $G:K\to\R$ (not necessarily in
$\NBV(K)$).

Note that a weak*-null family $(F_i)_{i\in I}$ in $\NBV(L)$ extends through a continuous increasing surjection $\phi:K\to L$ if and only if there exists a weak*-null family $(G_i)_{i\in I}$ in $\NBV(K)$ such that $\phi_*(G_i)=F_i$ for all $i\in I$. Our first lemma establishes the basic relations between a map $G:K\to\R$ and the map $F=\phi_*(G):L\to\R$, allowing us to work on the extendibility problem for weak*-null families in $\NBV(L)$.

\begin{lem}\label{thm:normBVestimaterightcont}
Let $\phi:K\to L$ be a continuous increasing surjection between compact lines, $G:K\to\R$ be a map and set $F=\phi_*(G)$. The following inequalities
hold:
\begin{equation}\label{eq:normaFGBV}
\max\Big\{\Vert F\Vert_{\BV},\sum_{t\in L}V\big(G\vert_{\phi^{-1}(t)}\big)\Big\}\le\Vert G\Vert_{\BV}\le\Vert F\Vert_{\BV}+2\sum_{t\in L}V\big(G\vert_{\phi^{-1}(t)}\big).
\end{equation}
Moreover, if $G$ is of bounded variation, then $G$ is right-continuous if and only if $F$ is right-continuous and $G\vert_{\phi^{-1}(t)}$ is right-continuous for all $t\in L$.
\end{lem}
\begin{proof}
It is easy to see that $\Vert F\Vert_{\BV}\le\Vert G\Vert_{\BV}$ and $\sum_{t\in L}V\big(G\vert_{\phi^{-1}(t)}\big)\le V(G)$, which implies
the first inequality in \eqref{eq:normaFGBV}. To prove the second inequality in \eqref{eq:normaFGBV}, let $P$ be a finite subset of $K$ and let us estimate $\vert G(\min K)\vert+V(G;P)$. Set $Q=\phi[P]=\{t_0,t_1,\ldots,t_n\}$ with $t_0<t_1<\cdots<t_n$ and write
\[P_i=P\cap\phi^{-1}(t_i),\quad\alpha_i=\min\phi^{-1}(t_i),\quad\beta_i=\max\phi^{-1}(t_i),\]
for $i=0,1,\ldots,n$. Without loss of generality we can assume that $\alpha_i$ and $\beta_i$ are in $P_i$ and that $t_0=\min L$. We have
\[V(G;P)\le\sum_{i=0}^nV\big(G\vert_{\phi^{-1}(t_i)}\big)+\sum_{i=0}^{n-1}\vert G(\alpha_{i+1})-G(\beta_i)\vert\]
and
\begin{align*}
\vert G(\alpha_{i+1})-G(\beta_i)\vert&\le\vert G(\beta_{i+1})-G(\beta_i)\vert+\vert G(\beta_{i+1})-G(\alpha_{i+1})\vert\\
&\le\vert F(t_{i+1})-F(t_i)\vert+V\big(G\vert_{\phi^{-1}(t_{i+1})}\big)
\end{align*}
so that:
\[V(G;P)\le V(F;Q)+V\big(G\vert_{\phi^{-1}(t_0)}\big)+2\sum_{i=1}^nV\big(G\vert_{\phi^{-1}(t_i)}\big).\]
Since $\vert G(\min K)\vert=\vert G(\alpha_0)\vert\le\vert F(\min L)\vert+V\big(G\vert_{\phi^{-1}(t_0)}\big)$, we get
\[\vert G(\min K)\vert+V(G;P)\le\Vert F\Vert_{\BV}+2\sum_{i=0}^nV\big(G\vert_{\phi^{-1}(t_i)}\big),\]
which proves the second inequality in \eqref{eq:normaFGBV}. As for the second part of the statement of the lemma, the only nontrivial fact to be proven
is that if $F$ is right-continuous and $G$ is of bounded variation then $G$ is right-continuous at the point $\beta_t=\max\phi^{-1}(t)$,
for every $t\in L\setminus L^+$. Since $G$ has bounded variation, the right-limit of $G$ at $\beta_t$ exists and we compute as follows:
\[\lim_{u\to\beta_t^+}G(u)=\lim_{s\to t^+}G(\beta_s)=\lim_{s\to t^+}F(s)=F(t)=G(\beta_t).\qedhere\]
\end{proof}

Now we need a criteria for establishing that a family $(G_i)_{i\in I}$ in $\NBV(K)$ is weak*-null when we already know that the family
$(F_i)_{i\in I}$ with $F_i=\phi_*(G_i)$ is weak*-null in $\NBV(L)$. To this aim, given a compact line $K$ and a nonempty closed interval $[t,s]\subset K$,
we define the {\em canonical retraction\/} of $K$ onto $[t,s]$ to be the retraction $R:K\to[t,s]$ that maps $[\min K,t]$ to $t$ and $[s,\max K]$ to $s$. Clearly, for any real-valued map $F$ on $K$, we have that $R_*(F)$ agrees with $F$ on $\left[t,s\right[$ and that $R_*(F)(s)$ is equal to $F(\max K)$.
\begin{lem}\label{thm:weaknullGicrit}
Let $\phi:K\to L$ be a continuous increasing surjection between compact lines and $(G_i)_{i\in I}$ be a family in $\NBV(K)$. For each $t\in L$, denote
by $R_t:K\to\phi^{-1}(t)$ the canonical retraction. We have that $(G_i)_{i\in I}$ is weak*-null in $\NBV(K)$ if and only if $(G_i)_{i\in I}$ is bounded, $\big(\phi_*(G_i)\big)_{i\in I}$ is weak*-null in $\NBV(L)$ and $\big((R_t)_*(G_i)\big)_{i\in I}$ is weak*-null in $\NBV\!\big(\phi^{-1}(t)\big)$ for all $t\in Q(\phi)$.
\end{lem}
\begin{proof}
Follows from the fact that the union of the image of the operator $\phi^*$ with the image of all of the operators $(R_t)^*$, $t\in Q(\phi)$, spans a dense subspace of $C(K)$ (see \cite[Lemma~2.2]{CTlines}).
\end{proof}

\begin{rem}\label{thm:remgenDA}
There is a simple construction that can be used to provide examples of continuous increasing surjections $\phi$ between compact lines with a prescribed set $Q(\phi)$. Namely, given a compact line $K$ and a family $(L_t)_{t\in K}$ of nonempty compact lines indexed on $K$, we have that the set $\widetilde K=\bigcup_{t\in K}\big(\{t\}\times L_t\big)$ endowed with the lexicographic order is a compact line and the first projection $\pi_1:\widetilde K\to K$ is a continuous increasing surjection. In particular, for an arbitrary subset $A$ of $K$, the {\em generalized double arrow space}
\[\DA(K;A)=\big(K\times\{0\}\big)\cup\big(A\times\{1\}\big)\]
endowed with the lexicographic order is a compact line and the first projection
$\pi_1:\DA(K;A)\to K$ is a continuous increasing surjection with $Q(\pi_1)=A$ .
\end{rem}

Let us now introduce the notion that will be used to write an equivalence for the property of extendibility of weak*-null families and which will be the central topic
of study in the rest of the paper.
\begin{defin}\label{thm:defc0ell1}
Let $Q$ be a set and $(F_i)_{i\in I}$ be a family of real-valued maps defined on $Q$. We say that $(F_i)_{i\in I}$ is {\em of type $c_0\ell_1$\/} if it is possible to write $F_i(t)$ as a sum
\begin{equation}\label{eq:Fitaitbit}
F_i(t)=a_{i,t}+b_{i,t}
\end{equation}
for certain families $(a_{i,t})_{i\in I,t\in Q}$, $(b_{i,t})_{i\in I,t\in Q}$ of real numbers such that
\begin{equation}\label{eq:aitc0}
\lim_{i\in I}a_{i,t}=0
\end{equation}
for all $t\in Q$ and:
\[\sup_{i\in I}\sum_{t\in Q}\vert b_{i,t}\vert<+\infty.\]
More specifically, given a bounded family $(k_i)_{i\in I}$ of nonnegative real numbers, we say that the family $(F_i)_{i\in I}$ is {\em of type $c_0\ell_1$ bounded by $(k_i)_{i\in I}$} if families of real numbers $(a_{i,t})_{i\in I,t\in Q}$ and $(b_{i,t})_{i\in I,t\in Q}$ exist satisfying \eqref{eq:Fitaitbit}, \eqref{eq:aitc0} and
$\sum_{t\in Q}\vert b_{i,t}\vert\le k_i$, for all $i\in I$.
\end{defin}
For convenience, we will say that a family $(F_i)_{i\in I}$ of real-valued maps defined on a set $Q$ has a certain property over a given subset $S$ of $Q$
if the family $(F_i\vert_S)_{i\in I}$ has the corresponding property. Clearly, the collection of all subsets $S$ of $Q$ such that $(F_i)_{i\in I}$
is of type $c_0\ell_1$ over $S$ is an ideal of the Boolean algebra of subsets of $Q$. Such ideal contains the ideal of finite subsets of $Q$ if the family $(F_i)_{i\in I}$ is pointwise bounded. Moreover, any family $(F_i)_{i\in I}$ is trivially of type $c_0\ell_1$ over the set:
\begin{equation}\label{eq:trivialset}
\big\{t\in Q:\lim_{i\in I}F_i(t)=0\big\}.
\end{equation}
A simple upper bound on the cardinality of the complement of \eqref{eq:trivialset} can be obtained for a family $(F_i)_{i\in I}$ of type $c_0\ell_1$.
\begin{prop}\label{thm:atmostcardI}
If $(F_i)_{i\in I}$ is a family of real-valued maps of type $c_0\ell_1$ defined over a set $Q$, then the cardinality of the
set
\begin{equation}\label{eq:tinQnotgotozero}
\big\{t\in Q:\big(F_i(t)\big)_{i\in I}\not\in c_0(I)\big\}
\end{equation}
is less than or equal to the cardinality of $I$.
\end{prop}
\begin{proof}
Simply note that given families $(a_{i,t})_{i\in I,t\in Q}$, $(b_{i,t})_{i\in I,t\in Q}$ as in Definition~\ref{thm:defc0ell1}, we have that the set
\eqref{eq:tinQnotgotozero} is contained in $\bigcup_{i\in I}\big\{t\in Q:b_{i,t}\ne0\big\}$ and that the set $\big\{t\in Q:b_{i,t}\ne0\big\}$ is countable
for all $i\in I$.
\end{proof}

Now we prove the equivalence between extendibility through $\phi$ of a weak*-null family and the property of being of type $c_0\ell_1$ over $Q(\phi)$.
\begin{teo}\label{thm:c0ell1equivextend}
Let $\phi:K\to L$ be a continuous increasing surjection between compact lines and $(F_i)_{i\in I}$ be a weak*-null family in $\NBV(L)$.
We have that $(F_i)_{i\in I}$ extends through $\phi$ if and only if $(F_i)_{i\in I}$ is of type $c_0\ell_1$ over $Q(\phi)$.
Moreover, if $(k_i)_{i\in I}$ is a bounded family of nonnegative real numbers and $(F_i)_{i\in I}$ is of type $c_0\ell_1$ bounded
by $(k_i)_{i\in I}$ over $Q(\phi)$, then a weak*-null family $(G_i)_{i\in I}$ in $\NBV(K)$ with $\phi_*(G_i)=F_i$ for all $i\in I$ can be chosen with
$\Vert G_i\Vert_{\BV}\le\Vert F_i\Vert_{\BV}+2k_i$, for all $i\in I$.
\end{teo}
\begin{proof}
Set $\alpha_t=\min\phi^{-1}(t)$ and $\beta_t=\max\phi^{-1}(t)$, for all $t\in L$.
Assume that $(G_i)_{i\in I}$ is a weak*-null family in $\NBV(K)$ such that $\phi_*(G_i)=F_i$ for all $i\in I$ and let us prove that $(F_i)_{i\in I}$ is of type $c_0\ell_1$ over $Q(\phi)$. For each $t\in Q(\phi)$, let $f_t:K\to[0,1]$ be a continuous map that equals $1$ on $[\min K,\alpha_t]$ and equals zero on $[\beta_t,\max K]$. For each $i\in I$ denote by $\mu_i\in M(K)$ the measure corresponding to $G_i$ and for each $t\in Q(\phi)$ set $a_{i,t}=\int_Kf_t\,\dd\mu_i$ and $b_{i,t}=F_i(t)-a_{i,t}$, so that $\lim_{i\in I}a_{i,t}=0$. We have
\[\vert b_{i,t}\vert=\big\vert G_i(\beta_t)-a_{i,t}\big\vert=\Bigg\vert\int_K\big(\chilow{[\min K,\beta_t]}-f_t\big)\,\dd\mu_i\Bigg\vert
\le\vert\mu_i\vert\big(\left]\alpha_t,\beta_t\right]\!\big)\]
so that $\sum_{t\in Q(\phi)}\vert b_{i,t}\vert\le\Vert\mu_i\Vert$ for all $i\in I$. Conversely, if $(F_i)_{i\in I}$ is of type $c_0\ell_1$ bounded by $(k_i)_{i\in I}$ over $Q(\phi)$, write $F_i(t)=a_{i,t}+b_{i,t}$ with $\lim_{i\in I}a_{i,t}=0$ for all $t\in Q(\phi)$ and
$\sum_{t\in Q(\phi)}\vert b_{i,t}\vert\le k_i$ for all $i\in I$ and define $G_i:K\to\R$ by setting $G_i\vert_{\left[\alpha_t,\beta_t\right[}\equiv a_{i,t}$
for all $i\in I$ and all $t\in Q(\phi)$ and $G_i(\beta_t)=F_i(t)$, for all $i\in I$ and all $t\in L$. Obviously $\phi_*(G_i)=F_i$ and Lemma~\ref{thm:normBVestimaterightcont} yields that $\Vert G_i\Vert_{\BV}\le\Vert F_i\Vert_{\BV}+2\sum_{t\in Q(\phi)}\vert b_{i,t}\vert\le\Vert F_i\Vert_{\BV}+2k_i$ and that $G_i$ is right-continuous, for all $i\in I$. Finally, Lemma~\ref{thm:weaknullGicrit} yields that $(G_i)_{i\in I}$ is weak*-null as for all $i\in I$ and $t\in Q(\phi)$ the map $(R_t)_*(G_i)\in\NBV\!\big([\alpha_t,\beta_t]\big)$ corresponds to the measure
$a_{i,t}\delta_{\alpha_t}+\big(F_i(\max L)-a_{i,t}\big)\delta_{\beta_t}$, where $\delta_s$ denotes the probability measure with support $\{s\}$.
\end{proof}

If $K$ is a compact line and $(F_i)_{i\in I}$ is a weak*-null family in $\NBV(K)$, then $(F_i)_{i\in I}$ is trivially of type $c_0\ell_1$ over $K^+$,
as $\lim_{i\in I}F_i(t)=0$ for all $t\in K^+$. Moreover, $(F_i)_{i\in I}$ is also necessarily of type $c_0\ell_1$ over $K^-$ and this is easily proven
by taking $a_{i,t}$ in \eqref{eq:Fitaitbit} to be the value of $F_i$ at the immediate predecessor of $t$, for all $t\in K^-\setminus\{\min K\}$.
Thus, when checking that $(F_i)_{i\in I}$ is of type $c_0\ell_1$ over a subset $Q$ of $K$, one has only to worry about the internal points of $K$ in $Q$.
We will see now that there are other ``small'' subsets of $Q$ that can be ignored when checking that a weak*-null family is of type $c_0\ell_1$ over $Q$.
We define below a notion of height for subsets of a compact line that is similar to the Cantor--Bendixson height of a topological space.
\begin{defin}
Let $K$ be a compact line and $A$ be a subset of $K$. We say that a point $t\in A$ is {\em internal to $A$ relatively to $K$\/}
if $t$ is neither the minimum nor the maximum of $K$ and if for all $s_1,s_2\in K$ with $s_1<t<s_2$ we have that $\left]s_1,t\right[\cap A\ne\emptyset$ and $\left]t,s_2\right[\cap A\ne\emptyset$. The {\em two-sided derivative\/}
of $A$ in $K$, denoted $\Der(A,K)$, is defined as the set of points of $A$ that are internal to $A$ relatively to $K$.
For any natural number $n$, we define recursively the {\em $n$-th two-sided derivative\/} of $A$ in $K$ by setting
$\Der_0(A,K)=A$ and $\Der_{n+1}(A,K)=\Der\!\big(\!\Der_n(A,K),K\big)$, for all $n\ge0$. We say that $A$ has {\em finite two-sided height\/} in $K$ if $\Der_n(A,K)=\emptyset$ for some natural number $n$ and in this case the least such natural number $n$ is called the {\em two-sided height\/} of $A$ in $K$.
\end{defin}
We observe that the {\em internal order\/} of $A$ defined in \cite{KK} is equal to the two-sided height of $A$ in $K$ minus $1$.
It follows from \cite[Lemma~2.6]{KK} that the collection of subsets of $K$ having finite two-sided height in $K$ is an ideal of the Boolean algebra of subsets of $K$.
\begin{prop}\label{thm:finiteheight}
Let $K$ be a compact line and $(F_i)_{i\in I}$ be a weak*-null family in $\NBV(K)$. If $A$ is a subset of $K$ having finite two-sided height in $K$, then $(F_i)_{i\in I}$ is of type $c_0\ell_1$ over $A$.
\end{prop}
\begin{proof}
Let $\phi:\DA(K;A)\to K$ denote the first projection, so that $\phi$ is a continuous increasing surjection with $Q(\phi)=A$ (see Remark~\ref{thm:remgenDA}). Since $Q(\phi)$ has finite two-sided height, it follows from \cite[Lemma~2.7]{KK} that the image of $\phi^*$ is complemented in $C\big(\!\DA(K;A)\big)$.
Hence any weak*-null family $(F_i)_{i\in I}$ in $\NBV(K)$ extends through $\phi$ and the conclusion follows from Theorem~\ref{thm:c0ell1equivextend}.
\end{proof}

If $(F_i)_{i\in I}$ is a family of real-valued maps defined on a set $Q$ and if $(F_i)_{i\in I}$ is of type $c_0\ell_1$ then, for a finite subset $S$ of $Q$, the sums $\sum_{t\in S}\vert F_i(t)\vert$, $i\in I$, are bounded, up to an element of $c_0(I)$, by a constant that is independent of $S$. In order to make this statement precise, we introduce the following notation.
\begin{defin}
Given families of real numbers $(x_i)_{i\in I}$ and $(y_i)_{i\in I}$, we will write
\[x_i\le y_i\mod c_0(I)\]
if there exists $(z_i)_{i\in I}$ in $c_0(I)$ such that $x_i\le y_i+z_i$, for all $i\in I$; equivalently, if the set $\big\{i\in I:x_i\ge y_i+\varepsilon\big\}$ is finite for all $\varepsilon>0$.
\end{defin}
Clearly, if a family $(F_i)_{i\in I}$ of real-valued maps defined on a set $Q$ is of type $c_0\ell_1$ bounded by $(k_i)_{i\in I}$ then
\begin{equation}\label{eq:finitesummodc0}
\sum_{t\in S}\vert F_i(t)\vert\le k_i\mod c_0(I),
\end{equation}
for every finite subset $S$ of $Q$. The converse holds when $Q$ is countable, as we show below.

\begin{prop}\label{thm:propTcountable}
Let $(F_i)_{i\in I}$ be a family of real-valued maps defined on a set $Q$ and assume that \eqref{eq:finitesummodc0} holds for every finite subset
$S$ of $Q$ and some bounded family $(k_i)_{i\in I}$ of nonnegative real numbers. If $Q$ is countable, then $(F_i)_{i\in I}$ is of type $c_0\ell_1$ bounded by $(k_i)_{i\in I}$.
\end{prop}
\begin{proof}
Write $Q=\bigcup_{n=1}^\infty Q_n$ as the union of an increasing sequence $(Q_n)_{n\ge1}$ of finite sets. Since $\sum_{t\in Q_n}\vert F_i(t)\vert\le k_i\mod c_0(I)$ for all $n$, we can find an increasing sequence $(I_n)_{n\ge1}$ of finite subsets of $I$ such that
\begin{equation}\label{eq:Fitki1n}
\sum_{t\in Q_n}\vert F_i(t)\vert<k_i+\frac1n
\end{equation}
for all $i\in I\setminus I_n$ and for all $n\ge1$. For $i\in I_1$, we set $a_{i,t}=F_i(t)$ and $b_{i,t}=0$ for all $t\in Q$. For each $n\ge1$ and for
$i\in I_{n+1}\setminus I_n$, since \eqref{eq:Fitki1n} holds, an elementary argument shows that we can write $F_i(t)=a_{i,t}+b_{i,t}$, for $t\in Q_n$, with
$\sum_{t\in Q_n}\vert a_{i,t}\vert\le\frac1n$ and $\sum_{t\in Q_n}\vert b_{i,t}\vert\le k_i$. For $i\in I_{n+1}\setminus I_n$ and $t\in Q\setminus Q_n$,
we set $a_{i,t}=F_i(t)$ and $b_{i,t}=0$. This completes the definition of $a_{i,t}$ and $b_{i,t}$ for all $i\in\bigcup_{n=1}^\infty I_n$ and all $t\in Q$.
Now for $i\in I\setminus\bigcup_{n=1}^\infty I_n$, we have that \eqref{eq:Fitki1n} holds for all $n\ge1$ and thus $\sum_{t\in Q}\vert F_i(t)\vert\le k_i$;
in this case we set $a_{i,t}=0$ and $b_{i,t}=F_i(t)$, for all $t\in Q$. We then have $F_i(t)=a_{i,t}+b_{i,t}$ for all $i\in I$, $t\in Q$ and
$\sum_{t\in Q}\vert b_{i,t}\vert\le k_i$, for all $i\in I$. Finally, given $t\in Q$, we have $t\in Q_n$ for $n\ge1$ sufficiently large and thus
$\vert a_{i,t}\vert\le\frac1n$ for all $i\in I$ outside of the finite set $I_n$. Hence $\lim_{i\in I}a_{i,t}=0$, concluding the proof.
\end{proof}

Now we prove some results that will yield estimates like \eqref{eq:finitesummodc0} for weak*-null families $(F_i)_{i\in I}$ in $\NBV(K)$.
\begin{lem}
Let $K$ be a compact line and $(F_i)_{i\in I}$ be a weak*-null family in $\NBV(K)$. For each $i\in I$, denote by $\mu_i\in M(K)$ the measure associated to $F_i$. Given $t,s\in K$ with $t<s$, we have that
\begin{gather}
\label{eq:ineqFits1}\vert F_i(t)\vert\le\vert\mu_i\vert\big(\left]t,s\right[\!\big)\mod c_0(I),\\
\label{eq:ineqFits2}\vert F_i(t)\vert+\vert F_i(s)\vert\le\vert\mu_i\vert\big(\left]t,s\right]\!\big)\mod c_0(I).
\end{gather}
\end{lem}
\begin{proof}
Pick a continuous function $f:K\to[0,1]$ which equals $1$ on $[\min K,t]$ and equals zero on $[s,\max K]$.
We have
\[\Bigg\vert F_i(t)-\int_Kf\,\dd\mu_i\Bigg\vert=\Bigg\vert\int_K\big(\chilow{[\min K,t]}-f\big)\,\dd\mu_i\Bigg\vert\le\vert\mu_i\vert\big(\left]t,s\right[\!\big),\]
for all $i\in I$ and $\lim_{i\in I}\int_Kf\,\dd\mu_i=0$, from which \eqref{eq:ineqFits1} follows. To prove \eqref{eq:ineqFits2}, we first note that
for any $u\in K$ with $u<s$ we have:
\begin{equation}\label{eq:ineqFitsaux}
\vert F_i(s)\vert\le\vert\mu_i\vert\big(\left]u,s\right]\!\big)\mod c_0(I);
\end{equation}
namely, this is obtained as in the proof of \eqref{eq:ineqFits1} above using a continuous function $f:K\to[0,1]$ which equals $1$ on $[\min K,u]$
and equals zero on $[s,\max K]$. Now if $\left]t,s\right[$ is empty, we have that $t$ is right-isolated and $\lim_{i\in I}F_i(t)=0$, so that
\eqref{eq:ineqFits2} follows from \eqref{eq:ineqFitsaux} with $u=t$.
Finally, if $\left]t,s\right[$ is nonempty, pick any $u\in\left]t,s\right[$ and obtain \eqref{eq:ineqFits2} by replacing $s$ with $u$ in \eqref{eq:ineqFits1} and
adding the resulting inequality to \eqref{eq:ineqFitsaux}.
\end{proof}

A subset $J$ of a compact line $K$ is said to be {\em convex\/} if for all $t,s\in J$ with $t\le s$ we have that $[t,s]$ is contained in $J$.
\begin{cor}\label{thm:corextremities}
Let $K$ be a compact line and $(F_i)_{i\in I}$ be a weak*-null family in $\NBV(K)$. Given a convex open subset $J$ of $K$ and $t\in J$, we have
\begin{gather}
\label{eq:ineqFtK1}\vert F_i(t)\vert\le\vert\mu_i\vert\big([\min K,t]\cap J\big)\mod c_0(I),\\
\label{eq:ineqFtK2}\vert F_i(t)\vert\le\vert\mu_i\vert\big(\left]t,\max K\right]\cap J\big)\mod c_0(I).
\end{gather}
\end{cor}
\begin{proof}
If $t$ is not the minimum of $J$, then \eqref{eq:ineqFtK1} follows by picking $s\in J$ with $s<t$ and by using \eqref{eq:ineqFits2} with the roles
of $t$ and $s$ interchanged. If $t$ is the minimum of $J$ then the fact that $J$ is open implies that $t$ is left-isolated. Thus
either $t$ is the minimum of $K$, in which case $F_i(t)=\mu_i\big(\{t\}\big)$, or $t$ has a right-isolated immediate predecessor $t^-\in K$
and $F_i(t)=F_i(t^-)+\mu_i\big(\{t\}\big)$, with $\lim_{i\in I}F_i(t^-)=0$. In either case, \eqref{eq:ineqFtK1} follows.
Now, to prove \eqref{eq:ineqFtK2}, use \eqref{eq:ineqFits1} with any $s\in J$ greater than $t$ if $t$ is not the maximum of $J$ and note
that \eqref{eq:ineqFtK2} holds trivially if $t$ is the maximum of $J$ since in this case $t$ is right-isolated and $\lim_{i\in I}F_i(t)=0$.
\end{proof}

\begin{prop}\label{thm:weakstarnullboundfinsums}
Let $K$ be a compact line, $(F_i)_{i\in I}$ be a weak*-null family in $\NBV(K)$, $J$ be a convex subset of $K$ and $S$ be a finite subset of $J$.
If $J$ has more than one point then
\begin{equation}\label{eq:ineqsumJnotopen}
\sum_{t\in S}\vert F_i(t)\vert\le\vert\mu_i\vert(J)\mod c_0(I)
\end{equation}
and if $J$ is open then
\begin{equation}\label{eq:ineqsumJopen}
\sum_{t\in S}\vert F_i(t)\vert\le\frac12\vert\mu_i\vert(J)\mod c_0(I),
\end{equation}
where $\mu_i\in M(K)$ denotes the measure associated to $F_i$.
\end{prop}
\begin{proof}
If $J$ is open and has a single point then this point is right-isolated and thus \eqref{eq:ineqsumJopen} holds trivially.
Assume then that $J$ has more than one point, in which case we can also assume that $S$ has more than one point.
Write $S=\{t_0,t_1,\ldots,t_n\}$ with $t_0<t_1<\cdots<t_n$ and $n\ge1$. Using \eqref{eq:ineqFits2} we get
\begin{equation}\label{eq:Ftjtjplus1}
\vert F_i(t_j)\vert+\vert F_i(t_{j+1})\vert\le\vert\mu_i\vert\big(\left]t_j,t_{j+1}\right]\!\big)\mod c_0(I),
\end{equation}
for all $j=0,1,\ldots,n-1$ and adding all the inequalities \eqref{eq:Ftjtjplus1} we obtain
\begin{equation}\label{eq:extreitiesonce}
\vert F_i(t_0)\vert+\vert F_i(t_n)\vert+2\sum_{j=1}^{n-1}\vert F_i(t_j)\vert\le\vert\mu_i\vert\big(\left]t_0,t_n\right]\!\big)\mod c_0(I),
\end{equation}
which proves \eqref{eq:ineqsumJnotopen}. Now if $J$ is open, Corollary~\ref{thm:corextremities} yields
\begin{gather*}
\vert F_i(t_0)\vert\le\vert\mu_i\vert\big([\min K,t_0]\cap J\big)\mod c_0(I),\\
\vert F_i(t_n)\vert\le\vert\mu_i\vert\big(\left]t_n,\max K\right]\cap J\big)\mod c_0(I)
\end{gather*}
and adding these inequalities with \eqref{eq:extreitiesonce} we obtain \eqref{eq:ineqsumJopen}.
\end{proof}

\begin{cor}\label{thm:corc0ell1oncountable}
If $K$ is a compact line and $(F_i)_{i\in I}$ is a weak*-null family in $\NBV(K)$ then $(F_i)_{i\in I}$ is of type $c_0\ell_1$ over every countable subset of $K$.
\end{cor}
\begin{proof}
Use Proposition~\ref{thm:propTcountable} and Proposition~\ref{thm:weakstarnullboundfinsums} with $J=K$.
\end{proof}

\begin{prop}\label{thm:countablesufficient}
Let $\phi:K\to L$ be a continuous increasing surjection between compact lines and $(F_i)_{i\in I}$ be a weak*-null family in $\NBV(L)$.
If the set
\begin{equation}\label{eq:tQphinotc0I}
\big\{t\in Q(\phi):\big(F_i(t)\big)_{i\in I}\not\in c_0(I)\big\}
\end{equation}
is countable then $(F_i)_{i\in I}$ extends through $\phi$.
\end{prop}
\begin{proof}
It follows from Corollary~\ref{thm:corc0ell1oncountable} and Theorem~\ref{thm:c0ell1equivextend}.
\end{proof}

When $I$ is countable, Proposition~\ref{thm:countablesufficient} and the combination of Proposition~\ref{thm:atmostcardI} with Theorem~\ref{thm:c0ell1equivextend} jointly recover the extension criterion presented in \cite[Proposition~2.5]{CTlines}, i.e., they yield that a weak*-null family $(F_i)_{i\in I}$ in $\NBV(L)$ extends through a continuous increasing surjection $\phi:K\to L$ if and only if the set \eqref{eq:tQphinotc0I} is countable. This does not hold in general when $I$ is uncountable, as the next example shows.
\begin{example}\label{exa:Victor}
Let $L_0$ be an arbitrary uncountable compact line and consider the product $L=L_0\times[0,1]$ ordered lexicographically, so that $L$ is also a compact line. Set $I=L_0\times\omega$ and for each $i=(s,n)\in I$ we let $F_i\in\NBV(L)$ be the characteristic function of the interval $\mathopen\big[(s,0),\big(s,\frac1{n+1}\big)\mathclose\big[$. We have that $(F_i)_{i\in I}$ is a bounded family in $\NBV(L)$ and using Lemma~\ref{thm:weaknullGicrit} with $\phi$ equal to the first projection of $L$ onto $L_0$ one easily obtains that $(F_i)_{i\in I}$ is weak*-null. The set
\[\big\{t\in L:\big(F_i(t)\big)_{i\in I}\not\in c_0(I)\big\}\]
is equal to $L_0\times\{0\}$ and thus \eqref{eq:tQphinotc0I} is uncountable if the continuous increasing surjection $\phi:K\to L$ is chosen
such that $Q(\phi)\cap\big(L_0\times\{0\}\big)$ is uncountable. Nevertheless, $(F_i)_{i\in I}$ is of type $c_0\ell_1$ over $L$, which can be seen for instance by noting that $\Der\!\big(L_0\times\{0\},L\big)=\emptyset$ and by using Proposition~\ref{thm:finiteheight}. Hence $(F_i)_{i\in I}$ extends through any continuous increasing surjection $\phi:K\to L$ by Theorem~\ref{thm:c0ell1equivextend}.
\end{example}

As we have observed earlier, the points in a compact line that are not internal are irrelevant when checking that a weak*-null family is of type $c_0\ell_1$ over a certain subset. It turns out that points with uncountable right character are also irrelevant. Here is the definition we need.
\begin{defin}
Given a compact line $K$, we define the {\em right character\/} (resp., {\em left character}) of a point $t$ in $K$ to be equal to $\omega$ if $t$ is right-isolated (resp., left-isolated) and to be equal to the least cardinal of a subset $A$ of $\left]t,\max K\right]$ with $\inf A=t$
(resp., of a subset $A$ of $\left[\min K,t\right[$ with $\sup A=t$), otherwise.
\end{defin}
It is easy to see that the character of a point in a compact line is precisely the maximum between its left character and its right character.

\begin{lem}\label{thm:uncountablerightcharacter}
Let $K$ be a compact line and $(F_i)_{i\in I}$ be a weak*-null family in $\NBV(K)$. If $t\in K$ has uncountable right character, then $\lim_{i\in I}F_i(t)=0$.
\end{lem}
\begin{proof}
Assume by contradiction that $\vert F_i(t)\vert\ge\varepsilon$ for some $\varepsilon>0$ and all $i$ in some infinite countable subset $I_0$ of $I$.
By the regularity of the measure $\mu_i\in M(K)$ associated to $F_i$, for each $i\in I_0$ we can find $t_i\in K$ with $t_i>t$ and
$\vert\mu_i\vert\big(\left]t,t_i\right]\!\big)<\frac\varepsilon2$. Since $t$ has uncountable right character, we have that $s=\inf\{t_i:i\in I_0\}$ is greater than $t$. Now pick a continuous function $f:K\to[0,1]$ which equals $1$ on $[\min K,t]$ and equals zero on $[s,\max K]$ and note that
\[\Bigg\vert\int_K f\,\dd\mu_i\Bigg\vert=\Bigg\vert F_i(t)+\int_{\left]t,s\right[}f\,\dd\mu_i\Bigg\vert\ge\vert F_i(t)\vert-\vert\mu_i\vert\big(\left]t,t_i\right]\big)>\frac\varepsilon2\]
for all $i\in I_0$, contradicting the fact that $\lim_{i\in I}\int_Kf\,\dd\mu_i=0$.
\end{proof}

In Proposition~\ref{thm:propBicountable} below we will obtain a useful equivalence for the property of being of type $c_0\ell_1$ for families of maps $(F_i)_{i\in I}$ satisfying
an estimate like \eqref{eq:finitesummodc0} and for which at most countably many maps $F_i$ are nonzero on any given point. As we saw in Proposition~\ref{thm:weakstarnullboundfinsums}, estimates like \eqref{eq:finitesummodc0} hold for weak*-null families in a compact line and we show now that the countability assumption in the statement of Proposition~\ref{thm:propBicountable} also holds for weak*-null families.
\begin{lem}\label{thm:Fitnonzerocountable}
If $K$ is a compact line and $(F_i)_{i\in I}$ is a weak*-null family in $\NBV(K)$ then the set $\big\{i\in I:F_i(t)\ne0\big\}$ is countable for any
$t\in K$.
\end{lem}
\begin{proof}
The result is trivial if $t$ is right-isolated and it follows directly from Lemma~\ref{thm:uncountablerightcharacter} if $t$ has uncountable right character. We therefore assume that $t$ is the infimum of a nonempty countable subset of $\left]t,\max K\right]$. In this case, there exists a decreasing sequence $(t_n)_{n\ge1}$ in $\left]t,\max K\right]$ converging to $t$. For each $n\ge1$, let $f_n:K\to[0,1]$ be a continuous function that equals $1$ on $[\min K,t]$ and equals zero on $[t_n,\max K]$. The sequence $(f_n)_{n\ge1}$ thus converges pointwise to the characteristic function of $[\min K,t]$ and it follows from the Dominated Convergence Theorem that $F_i(t)=\lim_{n\to+\infty}\int_Kf_n\,\dd\mu_i$, for all $i\in I$.
Hence $\big\{i\in I:F_i(t)\ne0\big\}$ is contained in the countable set $\bigcup_{n=1}^\infty\big\{i\in I:\int_Kf_n\,\dd\mu_i\ne0\big\}$.
\end{proof}

The following result is needed in the proof of Proposition~\ref{thm:propBicountable}. Recall that a family $(B^t)_{t\in Q}$ of sets is called {\em point countable\/} if for every $i$ the set $\big\{t\in Q:i\in B^t\big\}$ is countable.
\begin{lem}\label{thm:countablepartition}
If $(B^t)_{t\in Q}$ is a point countable family of countable sets then there exists a partition $Q=\bigcup_{\lambda\in\Lambda}Q_\lambda$ of $Q$ into countable sets such that $B^t\cap B^s=\emptyset$ for all $t\in Q_\lambda$, $s\in Q_\mu$ and all $\lambda,\mu\in\Lambda$ with $\lambda\ne\mu$.
\end{lem}
\begin{proof}
Define a reflexive and symmetric binary relation $R$ on $Q$ by stating that $(t,s)\in R$ if and only if either $t=s$ or $B^t\cap B^s\ne\emptyset$.
For each $t\in Q$, the set
\[R(t)=\big\{s\in Q:(t,s)\in R\big\}=\{t\}\cup\bigcup_{i\in B^t}\big\{s\in Q:i\in B^s\big\}\]
is countable. Let $\sim$ denote the transitive closure of $R$, i.e., $\sim$ is the binary relation on $Q$ defined by $t\sim s$ if and only if there exists $n\ge2$ and a sequence $(t_i)_{i=1}^n$ in $Q$ such that $t_1=t$, $t_n=s$ and $(t_i,t_{i+1})\in R$, for all $i=1,\ldots,n-1$. Clearly, $\sim$ is an equivalence relation on $Q$. Moreover, for each $t\in Q$, the equivalence class of $t$ with respect to $\sim$ is countable since such equivalence class is equal to the union $\bigcup_{n\in\omega}A_n$, where $A_0=\{t\}$ and $A_{n+1}=\bigcup_{s\in A_n}R(s)$, for all $n\ge0$. The desired partition of $Q$ is simply the partition induced by $\sim$.
\end{proof}

\goodbreak

\begin{prop}\label{thm:propBicountable}
Let $Q$ be a set, $(F_i)_{i\in I}$ be a family of real-valued maps defined on $Q$ and $(k_i)_{i\in I}$ be a bounded family of nonnegative real numbers.
Assume that $\sum_{t\in S}\vert F_i(t)\vert\le k_i\mod c_0(I)$, for every finite subset $S$ of $Q$ and that the set
$\big\{i\in I:F_i(t)\ne0\big\}$ is countable, for all $t\in Q$. The following conditions are equivalent:
\begin{itemize}
\item[(a)] $(F_i)_{i\in I}$ is of type $c_0\ell_1$;
\item[(b)] there exist disjoint sets $A$ and $B$ with $I\times Q=A\cup B$ such that $\lim_{i\in I}F_i(t)\chilow A(i,t)=0$ for all $t\in Q$ and
the set
\[B_i=\big\{t\in Q:(i,t)\in B\big\}\]
is countable, for all $i\in I$;
\item[(c)] $(F_i)_{i\in I}$ is of type $c_0\ell_1$ bounded by $(k_i)_{i\in I}$.
\end{itemize}
\end{prop}
\begin{proof}
Assuming (a), write $F_i(t)=a_{i,t}+b_{i,t}$ with $\lim_{i\in I}a_{i,t}=0$ for all $t\in Q$ and $\sup_{i\in I}\sum_{t\in Q}\vert b_{i,t}\vert<+\infty$ and define $A$ and $B$ by setting
\[B=\big\{(i,t)\in I\times Q:b_{i,t}\ne0\big\}\]
and $A=(I\times Q)\setminus B$. We have that $B_i$ is countable for all $i\in I$ and $\lim_{i\in I}F_i(t)\chilow A(i,t)=\lim_{i\in I}a_{i,t}\chilow A(i,t)=0$, for all $t\in Q$, proving (b). To conclude the proof of the proposition, we now assume the existence of sets $A$ and $B$ as in (b) and we
prove (c). By moving all pairs $(i,t)$ with $F_i(t)=0$ from $B$ to $A$, we can assume that $F_i(t)\ne0$ for all $(i,t)\in B$. With this assumption, we have that the set $B^t=\big\{i\in I:(i,t)\in B\big\}$ is countable, for all $t\in Q$. Moreover, the fact that each $B_i$ is countable means that the family $(B^t)_{t\in Q}$ is point countable and thus Lemma~\ref{thm:countablepartition} yields a partition $Q=\bigcup_{\lambda\in\Lambda}Q_\lambda$ of $Q$ into countable sets $Q_\lambda$ such that $B^t$ is disjoint from $B^s$ whenever $t\in Q_\lambda$, $s\in Q_\mu$ and $\lambda,\mu\in\Lambda$ are distinct. Proposition~\ref{thm:propTcountable} then yields that $(F_i)_{i\in I}$ is of type $c_0\ell_1$ bounded by $(k_i)_{i\in I}$ over $Q_\lambda$, for all $\lambda\in\Lambda$. As $(Q_\lambda)_{\lambda\in\Lambda}$ is a partition of $Q$, we obtain families of real numbers $(a_{i,t})_{i\in I,t\in Q}$,
$(b_{i,t})_{i\in I,t\in Q}$ with $F_i(t)=a_{i,t}+b_{i,t}$, for all $i\in I$, $t\in Q$, $\lim_{i\in I}a_{i,t}=0$, for all $t\in Q$ and
$\sum_{t\in Q_\lambda}\vert b_{i,t}\vert\le k_i$, for all $i\in I$ and all $\lambda\in\Lambda$. To establish that $(F_i)_{i\in I}$ is of type $c_0\ell_1$ bounded by $(k_i)_{i\in I}$, we define a new decomposition
\[F_i(t)=a'_{i,t}+b'_{i,t},\quad i\in I,\ t\in Q,\]
by setting $a'_{i,t}=F_i(t)\chilow A(i,t)+a_{i,t}\chilow B(i,t)$ and $b'_{i,t}=b_{i,t}\chilow B(i,t)$.
Clearly $\lim_{i\in I}a'_{i,t}=0$, for all $t\in Q$ and for each $i\in I$ we have that $B_i$ is contained in some $Q_\lambda$ and hence
$\sum_{t\in Q}\vert b'_{i,t}\vert=\sum_{t\in B_i}\vert b_{i,t}\vert\le\sum_{t\in Q_\lambda}\vert b_{i,t}\vert\le k_i$.
\end{proof}

\begin{cor}\label{thm:corsigmaideal}
Under the assumptions of Proposition~\ref{thm:propBicountable}, the collection of subsets of $Q$ over which the family $(F_i)_{i\in I}$ is of type $c_0\ell_1$ is closed under countable unions and it is therefore a $\sigma$-ideal of the Boolean algebra of subsets of $Q$. In particular, if $K$ is a compact line and $(F_i)_{i\in I}$ is a weak*-null family in $\NBV(K)$ then the collection of subsets of $K$ over which $(F_i)_{i\in I}$ is of type $c_0\ell_1$ is a $\sigma$-ideal of the Boolean algebra of subsets of $K$.
\end{cor}
\begin{proof}
For the first part, note that it is sufficient to establish closure under disjoint countable unions and use the equivalence between (a) and (b) of Proposition~\ref{thm:propBicountable}. For the second, note that the assumptions of Proposition~\ref{thm:propBicountable} are valid for weak*-null families due to
Proposition~\ref{thm:weakstarnullboundfinsums} and Lemma~\ref{thm:Fitnonzerocountable}.
\end{proof}

\begin{cor}\label{thm:corc0l1bound}
Let $K$ be a compact line, $(F_i)_{i\in I}$ be a weak*-null family in $\NBV(K)$, $J$ be a convex subset of $K$ and $Q$ be a subset of $J$.
For each $i\in I$, denote by $\mu_i\in M(K)$ the measure associated to $F_i$. If $(F_i)_{i\in I}$ is of type $c_0\ell_1$ over $Q$ and $J$ has more than one point then $(F_i)_{i\in I}$ is of type $c_0\ell_1$ over $Q$ bounded by $\big(\vert\mu_i\vert(J)\big)_{i\in I}$.
Moreover, if $(F_i)_{i\in I}$ is of type $c_0\ell_1$ over $Q$ and $J$ is open then $(F_i)_{i\in I}$ is of type $c_0\ell_1$ over $Q$ bounded by $\big(\frac12\vert\mu_i\vert(J)\big)_{i\in I}$.
\end{cor}
\begin{proof}
Follows from Lemma~\ref{thm:Fitnonzerocountable} and from Propositions~\ref{thm:weakstarnullboundfinsums} and \ref{thm:propBicountable}.
\end{proof}

\begin{cor}\label{thm:corJlambda}
Let $K$ be a compact line, $(F_i)_{i\in I}$ be a weak*-null family in $\NBV(K)$ and $(J_\lambda)_{\lambda\in\Lambda}$ be a family of pairwise disjoint
convex subsets of $K$ having more than one point. If $Q$ is a subset of $\bigcup_{\lambda\in\Lambda}J_\lambda$ and $(F_i)_{i\in I}$ is of type $c_0\ell_1$ over $Q\cap J_\lambda$ for all $\lambda\in\Lambda$ then $(F_i)_{i\in I}$ is of type $c_0\ell_1$ over $Q$.
\end{cor}
\begin{proof}
By Corollary~\ref{thm:corc0l1bound}, for each $\lambda\in\Lambda$, the family $(F_i)_{i\in I}$ is of type $c_0\ell_1$ over $Q\cap J_\lambda$ bounded by $\big(\vert\mu_i\vert(J_\lambda)\big)_{i\in I}$ and we can thus write $F_i(t)=a_{i,t}+b_{i,t}$ for $i\in I$, $t\in Q\cap J_\lambda$ with $\lim_{i\in I}a_{i,t}=0$ for all $t\in Q\cap J_\lambda$ and
\[\sum_{t\in Q\cap J_\lambda}\vert b_{i,t}\vert\le\vert\mu_i\vert(J_\lambda)\]
for all $i\in I$. Hence $\sum_{t\in Q}\vert b_{i,t}\vert\le\sum_{\lambda\in\Lambda}\vert\mu_i\vert(J_\lambda)\le\Vert\mu_i\Vert$ for all $i\in I$ and since $\sup_{i\in I}\Vert \mu_i\Vert<+\infty$ the conclusion follows.
\end{proof}

Using the results proven above we now show that if a $c_0(I)$-valued bounded operator $T$ defined on the subalgebra of $C(K)$ induced by a continuous
increasing surjection $\phi:K\to L$ admits a bounded extension to $C(K)$ then an extension can be obtained with norm at most twice the norm of $T$.
\begin{teo}
Let $\phi:K\to L$ be a continuous increasing surjection between compact lines and $T:C(L)\to c_0(I)$ be a bounded operator. If there exists a bounded
operator $S:C(K)\to c_0(I)$ with $S\circ\phi^*=T$ then there exists a bounded operator $S':C(K)\to c_0(I)$ with $S'\circ\phi^*=T$ and $\Vert S'\Vert\le2\Vert T\Vert$.
\end{teo}
\begin{proof}
Let $(F_i)_{i\in I}$ be the weak*-null family in $\NBV(L)$ corresponding to $T$. If a bounded operator $S:C(K)\to c_0(I)$ with $S\circ\phi^*=T$ exists then
$(F_i)_{i\in I}$ extends through $\phi$ and Theorem~\ref{thm:c0ell1equivextend} says that $(F_i)_{i\in I}$ is of type $c_0\ell_1$ over $Q(\phi)$.
By Corollary~\ref{thm:corc0l1bound}, $(F_i)_{i\in I}$ is then of type $c_0\ell_1$ over $Q(\phi)$ bounded by $\big(\frac12\Vert F_i\Vert_{\BV}\big)_{i\in I}$ and Theorem~\ref{thm:c0ell1equivextend} now yields a weak*-null family $(G_i)_{i\in I}$ in $\NBV(K)$ such that $\phi_*(G_i)=F_i$ and $\Vert G_i\Vert_{\BV}\le2\Vert F_i\Vert_{\BV}$, for all $i\in I$. The conclusion is obtained by letting $S':C(K)\to c_0(I)$ be the bounded operator corresponding to $(G_i)_{i\in I}$.
\end{proof}

\subsection{A characterization of weak*-null families}\label{sub:weakstarnull}

As an offshoot of the theory developed in this section we obtain a simple criterion that can be used to check if a family $(F_i)_{i\in I}$ in $\NBV(K)$ is weak*-null.
We start with a lemma that allows us to reduce the general case to the case of a metrizable (and thus separable) compact line.
\begin{lem}\label{thm:sufficesphiZ}
Given a compact line $K$, we have that a family $(F_i)_{i\in I}$ in $\NBV(K)$ is weak*-null if and only if for every closed subset $Z$ of $[0,1]$ and every continuous increasing surjection $\phi:K\to Z$, the family $\big(\phi_*(F_i)\big)_{i\in I}$ is weak*-null in $\NBV(Z)$.
\end{lem}
\begin{proof}
It is sufficient to show that for every $f\in C(K)$ there exists a closed subset $Z$ of $[0,1]$ and a continuous increasing surjection $\phi:K\to Z$ such that $f$ belongs to the image of $\phi^*:C(Z)\to C(K)$. By \cite[Proposition~3.2]{Kubis}, the subset of $C(K)$ consisting of continuous increasing functions
spans a dense subspace of $C(K)$ and therefore there exists a sequence $(f_n)_{n\ge1}$ of continuous increasing functions $f_n:K\to\R$ such that $f$ belongs
to the closure of the linear span of the set $\{f_n:n\ge1\}$. Moreover, we can assume that the range of $f_n$ is contained in $[0,1]$, for all $n\ge1$.
Set $\phi=\sum_{n=1}^\infty\frac1{2^n}f_n$ and let $Z\subset[0,1]$ denote the range of $\phi$. To prove that $f$ belongs to the image of $\phi^*$,
simply note that for all $t,s\in K$, if $\phi(t)=\phi(s)$ then $f_n(t)=f_n(s)$ for all $n\ge1$ and hence $f(t)=f(s)$.
\end{proof}

For a zero-dimensional compact line $K$ we have that a family $(F_i)_{i\in I}$ in $\NBV(K)$ is weak*-null if and only if it is bounded and
$\lim_{i\in I}F_i(t)=0$, for all $t\in K^+$. This follows from the fact that if $K$ is zero-dimensional then $\big\{\chilow{[\min K,t]}:t\in K^+\big\}$
spans a dense subspace of $C(K)$. The idea of the proof of our criterion for weak*-nullity below is to write an arbitrary (separable) compact line as the image of a continuous increasing map $\phi$ defined in a zero-dimensional compact line and then transfer the problem to the zero-dimensional domain of $\phi$ in which the simpler criterion for weak*-nullity applies. To obtain such map $\phi$, we will need the following observation: if $A$ is a dense subset of a compact line $K$ then the compact line $\DA(K;A)$ (recall Remark~\ref{thm:remgenDA}) is zero-dimensional. This follows for instance from the fact that a compact line is zero-dimensional if and only if its subset of right-isolated points is dense.
\begin{teo}\label{thm:charactweakstarnull}
Given a compact line $K$, we have that a family $(F_i)_{i\in I}$ in $\NBV(K)$ is weak*-null if and only if all of the following conditions hold:
\begin{itemize}
\item[(a)] $(F_i)_{i\in I}$ is bounded in $\NBV(K)$;
\smallskip
\item[(b)] $\lim_{i\in I}F_i(t)=0$, for all $t\in K^+$;
\smallskip
\item[(c)] there exists $k\ge0$ such that $\sum_{t\in S}\vert F_i(t)\vert\le k\mod c_0(I)$ for every finite subset $S$ of $K$.
\end{itemize}
\end{teo}
\begin{proof}
If $(F_i)_{i\in I}$ is weak*-null then (a) and (b) clearly hold and (c) follows from Proposition~\ref{thm:weakstarnullboundfinsums}. Conversely, assume
that (a), (b) and (c) hold. We first handle the case where $K$ is separable. Let $A$ be a countable dense subset of $K$ and consider the first projection
$\phi:\DA(K;A)\to K$. As noted above, the fact that $A$ is dense in $K$ implies that $\DA(K;A)$ is zero-dimensional.
We will now obtain a weak*-null family $(G_i)_{i\in I}$ in $\NBV\!\big(\!\DA(K;A)\!\big)$ with $\phi_*(G_i)=F_i$ for all $i\in I$, from which the fact that $(F_i)_{i\in I}$ is weak*-null will follow from the weak*-continuity of $\phi_*$.

In order to construct $(G_i)_{i\in I}$, note first that
from (c) and Proposition~\ref{thm:propTcountable} it follows that $(F_i)_{i\in I}$ is of type $c_0\ell_1$ over $A$. Write
$F_i(t)=a_{i,t}+b_{i,t}$ for $i\in I$, $t\in A$ with $\lim_{i\in I}a_{i,t}=0$ for all $t\in A$ and $\sup_{i\in I}\sum_{t\in A}\vert b_{i,t}\vert<+\infty$.
For each $i\in I$, define the map $G_i:\DA(K;A)\to\R$ by setting $G_i(t,0)=a_{i,t}$ and $G_i(t,1)=F_i(t)$ for all $t\in A$ and $G_i(t,0)=F_i(t)$ for all $t\in K\setminus A$, so that $\phi_*(G_i)=F_i$. It follows from (a) and Lemma~\ref{thm:normBVestimaterightcont} that $(G_i)_{i\in I}$ is a bounded family in $\NBV\!\big(\!\DA(K;A)\!\big)$.
As $\DA(K;A)$ is zero-dimensional, to check that $(G_i)_{i\in I}$ is weak*-null, pick $s\in\DA(K;A)^+$ and note that either $s=(t,0)$ with $t\in A$, in which case $\lim_{i\in I}G_i(s)=\lim_{i\in I}a_{i,t}=0$, or $s=\max\phi^{-1}(t)$ with $t\in K^+$, in which case $\lim_{i\in I}G_i(s)=\lim_{i\in I}F_i(t)=0$ by (b). This concludes the proof of the theorem in the separable case.

The general case now follows using Lemma~\ref{thm:sufficesphiZ}. Namely, if $(F_i)_{i\in I}$ satisfies (a), (b) and (c) and $\phi:K\to Z$ is a continuous increasing surjection with $Z$ a closed subset of $[0,1]$, then the family $\big(\phi_*(F_i)\big)_{i\in I}$ is bounded in $\NBV(Z)$ and satisfies
$\lim_{i\in I}\phi_*(F_i)(t)=0$ for every $t\in Z^+$ and
\[\sum_{t\in S}\vert\phi_*(F_i)(t)\vert\le k\mod c_0(I)\]
for every finite subset $S$ of $Z$. Hence $\big(\phi_*(F_i)\big)_{i\in I}$ is weak*-null in $\NBV(Z)$ by the separable case already proven.
\end{proof}

\end{section}

\begin{section}{Separably determined compact lines}
\label{sec:separablydetermined}

If $I$ is a countable set then Proposition~\ref{thm:atmostcardI} and Corollary~\ref{thm:corc0ell1oncountable} imply that a weak*-null family $(F_i)_{i\in I}$ in $\NBV(K)$ is of type $c_0\ell_1$ over a subset $Q$ of a compact line $K$ if and only if $\lim_{i\in I}F_i(t)=0$ for all but countably many points $t\in Q$. However, if $I$ is uncountable, we have seen in Example~\ref{exa:Victor} that the condition $\lim_{i\in I}F_i(t)=0$ might fail for uncountably many $t\in Q$ for a weak*-null family $(F_i)_{i\in I}$ of type $c_0\ell_1$ over $Q$. Let us now investigate how frequently such failure may occur within the intersection of $Q$ with a separable subset of $K$. We need a preparatory result.
\begin{lem}\label{thm:lemaHsepJcountable}
Let $K$ be a compact line, $(F_i)_{i\in I}$ be a weak*-null family in $\NBV(K)$ and $H$ be a closed separable subset of $K$. There exists a countable subset $I_0$ of $I$ such that $F_i$ is identically zero on $H\setminus H^+$ for all $i\in I\setminus I_0$.
\end{lem}
\begin{proof}
For each $i\in I$, denote by $\mu_i\in M(K)$ the measure associated to $F_i$. Let $D$ be a countable dense subset of $H$ and for each $t,s\in D$ with $t<s$ choose a continuous function $f_{ts}:K\to[0,1]$ that equals $1$ on $[\min K,t]$ and equals zero on $[s,\max K]$. We will show that $F_i$ is identically zero on $H\setminus H^+$ for every $i\in I$ outside of the countable set
\[I_0=\!\bigcup_{\substack{t,s\in D\\t<s}}\Big\{i\in I:\int_Kf_{ts}\,\dd\mu_i\ne0\Big\}.\]
To this aim, fix $u\in H\setminus H^+$ and $i\in I\setminus I_0$. Given $\varepsilon>0$, the regularity of $\mu_i$ yields $v\in K$ with
$v>u$ and $\vert\mu_i\vert\big(\left]u,v\right[\!\big)<\varepsilon$. Since $u$ is not right-isolated in $H$,
the intersection of $\left]u,v\right[$ with $H$ is an infinite open subset of $H$ and we can thus find $t,s\in D\cap\left]u,v\right[$ with $t<s$.
Now $\int_Kf_{ts}\,\dd\mu_i=0$ and hence:
\[\big\vert F_i(u)\big\vert=\Bigg\vert\int_K\big(\chilow{[\min K,u]}-f_{ts}\big)\,\dd\mu_i\Bigg\vert
\le\vert\mu_i\vert\big(\left]u,v\right[\!\big)<\varepsilon.\qedhere\]
\end{proof}

\begin{teo}\label{thm:teoHseparable}
Let $K$ be a compact line, $Q$ be a subset of $K$ and $(F_i)_{i\in I}$ be a weak*-null family in $\NBV(K)$.
If $(F_i)_{i\in I}$ is of type $c_0\ell_1$ over $Q$, then the set
\begin{equation}\label{eq:QinterHsetminusHplus}
\big\{t\in Q\cap(H\setminus H^+):\big(F_i(t)\big)_{i\in I}\not\in c_0(I)\big\}
\end{equation}
is countable for every closed separable subset $H$ of $K$ and, moreover, the set
\[\big\{t\in Q\cap H:\big(F_i(t)\big)_{i\in I}\not\in c_0(I)\big\}\]
is countable for every closed separable convex subset $H$ of $K$.
\end{teo}
\begin{proof}
For the first part of the statement, pick a countable subset $I_0$ of $I$ as in Lemma~\ref{thm:lemaHsepJcountable} and use Proposition~\ref{thm:atmostcardI} for the family $(F_i)_{i\in I_0}$ which is of type $c_0\ell_1$ over $Q\cap(H\setminus H^+)$. For the second part, simply note that if $H$ is convex then the only point of $H^+$ that might not be in $K^+$ is the maximum of $H$.
\end{proof}
Theorem~\ref{thm:teoHseparable} no longer holds if $H\setminus H^+$ is replaced with $H$ in \eqref{eq:QinterHsetminusHplus}. For instance, consider
the weak*-null family $(F_i)_{i\in I}$ of type $c_0\ell_1$ over the compact line $L=L_0\times[0,1]$ defined in Example~\ref{exa:Victor} with $L_0=[0,1]$ and take the separable
closed subset $H$ of $L$ to be $L_0\times\{0,1\}$.

\smallskip

One could hope for the validity of the converse of the first part of the statement of Theorem~\ref{thm:teoHseparable}, as that would yield a nice generalization of \cite[Proposition~2.5]{CTlines}. However, such converse does not hold in general. In fact, we will study now the class of compact lines for which the converse does hold and later in Section~\ref{sec:notsepdet} we will see that not all compact lines belong to such class.
\begin{defin}\label{thm:defseparablydetermined}
A compact line $K$ is said to be {\em separably determined\/} if for every weak*-null family $(F_i)_{i\in I}$ in $\NBV(K)$ and every subset $Q$ of $K$
the following condition holds: if the set \eqref{eq:QinterHsetminusHplus} is countable for every
closed separable subset $H$ of $K$ then $(F_i)_{i\in I}$ is of type $c_0\ell_1$ over $Q$.
\end{defin}

Let us show now that the class of separably determined compact lines can also be defined in terms of an equivalence between the $c_0$-extension property
and the $c_0(I)$-extension property for continuous increasing surjections.
\begin{teo}
Given a compact line $L$, the following conditions are equivalent:
\begin{itemize}
\item[(a)] $L$ is separably determined;
\item[(b)] for every compact line $K$, every continuous increasing surjection $\phi:K\to L$
and every set $I$, the map $\phi$ has the $c_0$-extension property if and only if $\phi$ has the $c_0(I)$-extension property.
\end{itemize}
\end{teo}
\begin{proof}
Assume (a). By Theorem~\ref{thm:c0ell1equivextend}, to prove (b) it suffices to prove that if $\phi$ has the $c_0$-extension property then any weak*-null family $(F_i)_{i\in I}$ in $\NBV(L)$ is of type $c_0\ell_1$ over $Q(\phi)$. Since $L$ is separably determined, this amounts to verifying that if $\phi$ has the $c_0$-extension property then for every separable closed subset $H$ of $K$, the set \eqref{eq:QinterHsetminusHplus} with $Q=Q(\phi)$ is countable. To this aim, pick a countable subset $I_0$
of $I$ as in Lemma~\ref{thm:lemaHsepJcountable} and note that, since $\phi$ has the $c_0$-extension property, the countable weak*-null family $(F_i)_{i\in I_0}$ is of type $c_0\ell_1$ over $Q(\phi)$ and hence, by Proposition~\ref{thm:atmostcardI}, the set $\big\{t\in Q(\phi):\big(F_i(t)\big)_{i\in I_0}\not\in c_0(I_0)\big\}$ is countable.

Now assume (b) and fix a subset $Q$ of $L$ and a weak*-null family $(F_i)_{i\in I}$ in $\NBV(L)$ such that the set \eqref{eq:QinterHsetminusHplus} is countable for every closed separable subset $H$ of $L$. We show that $(F_i)_{i\in I}$ is of type $c_0\ell_1$ over $Q$. Obviously, it is sufficient to check that $(F_i)_{i\in I}$ is of type $c_0\ell_1$ over $Q'=\big\{t\in Q:\big(F_i(t)\big)_{i\in I}\not\in c_0(I)\big\}$. Set $K=\DA(L;Q')$ and denote by
$\phi:K\to L$ the first projection, so that $\phi$ is a continuous increasing surjection and $Q(\phi)=Q'$ (Remark~\ref{thm:remgenDA}). According to \cite[Theorem~2.6]{CTlines}, $\phi$ has the $c_0$-extension property if and only if for every separable $G_\delta$ subset $A$ of $L$ such that $\Der(A,L)=A$ we have that $A\cap Q(\phi)$ is countable; for such a subset $A$ of $L$, we have that the closure $H$ of $A$ is a separable closed subset of $L$ and that $A\subset H\setminus H^+$, so that our assumptions on $(F_i)_{i\in I}$ imply that $A\cap Q(\phi)=A\cap Q'$ is countable. Hence $\phi$ has the $c_0$-extension property and also the $c_0(I)$-extension property, so that $(F_i)_{i\in I}$ is of type $c_0\ell_1$ over $Q(\phi)=Q'$, by Theorem~\ref{thm:c0ell1equivextend}.
\end{proof}

We will now prove that a compact line with a small set of ``bad points'' is separably determined. In this context, the ``bad points'' are the internal points of countable right character that do not have a separable adjacency in the sense defined below.
\begin{defin}
A point $t$ of a compact line $K$ is said to have a {\em separable adjacency\/} in $K$ if $t$ belongs to a separable convex subset of $K$ with more than one point; equivalently, if there exists $s\in\left]t,\max K\right]$ such that $[t,s]$ is separable or there exists $s\in\left[\min K,t\right[$ such that $[s,t]$ is separable.
\end{defin}

\begin{lem}\label{thm:propsepadj}
Let $K$ be a compact line, $Q$ be a subset of $K$ and $(F_i)_{i\in I}$ be a weak*-null family in $\NBV(K)$. If the set
\[\big\{t\in Q\cap H:\big(F_i(t)\big)_{i\in I}\not\in c_0(I)\big\}\]
is countable for every closed separable convex subset $H$ of $K$ then
$(F_i)_{i\in I}$ is of type $c_0\ell_1$ over the set of all points of $Q$ that have a separable adjacency in $K$.
\end{lem}
\begin{proof}
Define a binary relation $\sim$ on $K$ by stating that $t\sim s$ if and only if the closed interval with extremities $t$ and $s$ is separable.
It is readily checked that $\sim$ is an equivalence relation on $K$ whose equivalence classes are convex. Moreover, a point of $K$ has a separable adjacency in $K$ if and only if its equivalence class has more than one point. If $J$ is an equivalence class of $\sim$ with more than one point then it is easily proven using transfinite recursion and the fact
that intervals with extremities in $J$ are separable that $J$ is a disjoint union of separable convex subsets of $K$ having more than one point.
Hence the set of all points of $K$ with a separable adjacency can be written as a disjoint union $\bigcup_{\lambda\in\Lambda}J_\lambda$ of a family $(J_\lambda)_{\lambda\in\Lambda}$ of separable convex subsets of $K$ having more than one point. By our assumptions, the set
$\big\{t\in Q\cap\overline{J_\lambda}:\big(F_i(t)\big)_{i\in I}\not\in c_0(I)\big\}$ is countable and thus by Corollary~\ref{thm:corc0ell1oncountable} the family $(F_i)_{i\in I}$ is of type $c_0\ell_1$ over $Q\cap J_\lambda$, for all $\lambda\in\Lambda$. The conclusion now follows from Corollary~\ref{thm:corJlambda}.
\end{proof}

\begin{teo}
Let $K$ be a compact line. If the set of internal points of $K$ that have countable right character and do not have a separable adjacency in $K$ can be covered by countably many
sets of finite two-sided height in $K$ then $K$ is separably determined.
\end{teo}
\begin{proof}
It follows from Lemmas~\ref{thm:uncountablerightcharacter} and \ref{thm:propsepadj}, Corollary~\ref{thm:corsigmaideal} and Proposition~\ref{thm:finiteheight}.
\end{proof}

We now prove a few closure properties for the class of separably determined compact lines. First let us show that the property of being separably determined is hereditary to closed subsets. We need a lemma.
\begin{lem}\label{thm:lemaHplusminusKplus}
If $K$ is a separable compact line and $H$ is a closed subset of $K$ then the set $H^+\setminus K^+$ is countable.
\end{lem}
\begin{proof}
For each $t\in H^+\setminus\{\max H\}$, denote by $\lambda(t)$ the immediate successor of $t$ in $H$ and note that the intervals
$\left]t,\lambda(t)\right[$ with $t\in H^+\setminus\big(K^+\cup\{\max H\}\big)$ are open, nonempty and pairwise disjoint. Hence the separability of $K$ implies that
there can be only countably many such intervals.
\end{proof}

\begin{prop}
If $K$ is a separably determined compact line and $L$ is a closed subset of $K$ then $L$ is a separably determined compact line.
\end{prop}
\begin{proof}
Let $Q$ be a subset of $L$ and $(F_i)_{i\in I}$ be a weak*-null family in $\NBV(L)$ such that the set $\big\{t\in Q\cap(H\setminus H^+):\big(F_i(t)\big)_{i\in I}\not\in c_0(I)\big\}$ is countable for every closed separable subset $H$ of $L$. Let us prove that $(F_i)_{i\in I}$ is of type $c_0\ell_1$ over $Q$. Denote by $\iota:L\to K$ the inclusion map and set $G_i=\iota_*(F_i)$ for all $i\in I$. We have that $(G_i)_{i\in I}$ is a weak*-null family in $\NBV(K)$ and it is readily seen that $G_i$ extends $F_i$, for all $i\in I$. Since $K$ is separably determined, the conclusion will follow if we show that for a closed separable subset $H$ of $K$ the set $\big\{t\in Q\cap(H\setminus H^+):\big(G_i(t)\big)_{i\in I}\not\in c_0(I)\big\}$ is countable.
To this aim, note first that $\lim_{i\in I}G_i(t)=0$ for $t\in K\setminus L$; namely, if $[\min K,t]\cap L=\emptyset$ then $G_i(t)=0$ for all $i\in I$ and otherwise $G_i(t)=F_i(s)$ for all $i\in I$, where $s=\max\!\big([\min K,t]\cap L\big)$ is in $L^+$. Since a separable compact line is hereditarily separable (\cite[Lemma~2.9]{CTlines}) we have that $H\cap L$ is a separable closed subset of $L$ and thus by our assumptions on $(F_i)_{i\in I}$ we have that the set
$\big\{t\in Q\cap\big((H\cap L)\setminus(H\cap L)^+\big):\big(F_i(t)\big)_{i\in I}\not\in c_0(I)\big\}$ is countable. To conclude the proof now apply Lemma~\ref{thm:lemaHplusminusKplus} to the separable compact line $H$ to see that the set $(H\cap L)^+\setminus H^+$ is countable.
\end{proof}

We note that the image of a continuous increasing map with separably determined domain is not always separably determined since every compact
line with no internal points is separably determined and every compact line $K$ is the image of a continuous increasing map whose domain has no internal points
(for instance, the first projection of $\DA(K;K)$ --- see Remark~\ref{thm:remgenDA}). Moreover, not every compact line is separably determined, as we will see in Section~\ref{sec:notsepdet}.

Our next result states that the class of separably determined compact lines is closed under the lexicographic-ordered product construction described in Remark~\ref{thm:remgenDA}.
\begin{teo}\label{thm:teolexprodsepdet}
Let $\phi:K\to L$ be a continuous increasing surjection between compact lines. If $L$ is separably determined and $\phi^{-1}(t)$ is separably determined for all $t\in L$ then $K$ is separably determined.
\end{teo}
\begin{proof}
Let $Q$ be a subset of $K$ and $(F_i)_{i\in I}$ be a weak*-null family in $\NBV(K)$ such that the set $\big\{s\in Q\cap(H\setminus H^+):\big(F_i(s)\big)_{i\in I}\not\in c_0(I)\big\}$ is countable for every closed separable subset $H$ of $K$. We show that $(F_i)_{i\in I}$ is of type $c_0\ell_1$ over $Q$. For each $t\in L$, denote by $R_t$ the canonical retraction of $K$ onto $\phi^{-1}(t)$ and set $F^t_i=(R_t)_*(F_i)$ for all $i\in I$ and $t\in L$, so that $(F^t_i)_{i\in I}$ is a weak*-null family in $\NBV\!\big(\phi^{-1}(t)\big)$ for all $t\in L$. Given $t\in L$, the maps $F^t_i$ and $F_i$ are equal on $\phi^{-1}(t)\setminus\{\max\phi^{-1}(t)\}$ for all $i\in I$ and thus we have that the set $\big\{s\in Q\cap\phi^{-1}(t)\cap(H\setminus H^+):\big(F^t_i(s)\big)_{i\in I}\not\in c_0(I)\big\}$ is countable for every closed separable subset $H$ of $\phi^{-1}(t)$. Since $\phi^{-1}(t)$ is separably determined, we obtain that $(F^t_i)_{i\in I}$ and hence $(F_i)_{i\in I}$ is of type $c_0\ell_1$ over $Q\cap\phi^{-1}(t)$ for all $t\in L$. Corollary~\ref{thm:corJlambda} then yields that $(F_i)_{i\in I}$ is of type $c_0\ell_1$ over $Q\cap\bigcup_{t\in Q(\phi)}\phi^{-1}(t)$. To conclude the proof, we need to show that $(F_i)_{i\in I}$ is of type $c_0\ell_1$ over the set:
\[Q'=Q\setminus\bigcup_{t\in Q(\phi)}\phi^{-1}(t)=\big\{s\in Q:\phi^{-1}\big(\phi(s)\big)=\{s\}\big\}.\]
For each $i\in I$, set $G_i=\phi_*(F_i)$, so that $(G_i)_{i\in I}$ is a weak*-null family in $\NBV(L)$. We have that $\phi\vert_{Q'}$ is injective and that $G_i\big(\phi(s)\big)=F_i(s)$ for all $s\in Q'$ and therefore to prove that $(F_i)_{i\in I}$ is of type $c_0\ell_1$ over $Q'$ it is sufficient to prove that $(G_i)_{i\in I}$ is of type $c_0\ell_1$ over $\phi[Q']$. Since $L$ is separably determined, the latter will be established once we check that the
set
\begin{equation}\label{eq:tinphiQprime}
\big\{t\in\phi[Q']\cap(H\setminus H^+):\big(G_i(t)\big)_{i\in I}\not\in c_0(I)\big\}
\end{equation}
is countable for an arbitrary closed separable subset $H$ of $L$. Let $\widetilde H$ be the closure of a countable subset of $K$ whose image under $\phi$ is a countable dense subset of $H$. We have that $\widetilde H$ is a closed separable subset of $K$ and $\phi[\widetilde H]=H$. From the separability of $\widetilde H$, it follows that the set
\begin{equation}\label{eq:sinQ}
\big\{s\in Q\cap(\widetilde H\setminus\widetilde H^+):\big(F_i(s)\big)_{i\in I}\not\in c_0(I)\big\}
\end{equation}
is countable. Given $t\in\phi[Q']\cap H$, we have that $t=\phi(s)$ for a unique $s\in K$; moreover, $s$ is in $Q'\cap\widetilde H$, $G_i(t)=F_i(s)$ for all $i\in I$ and $t$ belongs to $H^+$ if and only if $s$ belongs to $\widetilde H^+$. Hence \eqref{eq:tinphiQprime} is contained in the image of \eqref{eq:sinQ} under $\phi$ and we are done.
\end{proof}

\begin{cor}\label{thm:corprodsepdet}
The cartesian product of two separably determined compact lines endowed with the lexicographic order is a separably determined compact line.
\end{cor}
\begin{proof}
Apply Theorem~\ref{thm:teolexprodsepdet} with $\phi$ equal to the first projection.
\end{proof}

\end{section}

\begin{section}{Examples of compact lines that are not separably determined}
\label{sec:notsepdet}

We have seen in Corollary~\ref{thm:corprodsepdet} that the class of separably determined compact lines is closed under finite products, where products are endowed with the lexicographic order. In this section we show that the situation is different for infinite products.

In what follows, we fix an arbitrary uncountable compact line $L$ and we denote by $K=L^\omega$ the set of all sequences in $L$ endowed with the lexicographic order. We have that $K$ is a compact line and the main goal of this section is to prove that $K$ is not separably determined.
Let $K^0$ (resp., $K^1$) denote the subset of $K$ consisting of sequences $t$ such that $t_n=\min L$ (resp., $t_n=\max L$) for all but finitely many $n\in\omega$. For any set $R$, let $R^{<\omega}=\bigcup_{n\in\omega}R^n$ be the set of all finite sequences in $R$. Set $\mathcal E=L^{<\omega}$ and for each $\epsilon\in\mathcal E$ let $C_\epsilon=\big\{t\in K:\epsilon\subset t\big\}$ denote the set of sequences in $L$ that extend $\epsilon$ and $F_\epsilon$ denote the characteristic function of $C_\epsilon\setminus\{\max C_\epsilon\}$. We have that $C_\epsilon$ is a closed convex subset of $K$ and that $F_\epsilon$ is an element of $\NBV(K)$ with $\Vert F_\epsilon\Vert_{\BV}=2$, for all $\epsilon\in\mathcal E$. We will use the family $(F_\epsilon)_{\epsilon\in\mathcal E}$ to establish that $K$ is not separably determined. The first step is to check that it is weak*-null.
\begin{lem}
The family $(F_\epsilon)_{\epsilon\in\mathcal E}$ is weak*-null in $\NBV(K)$.
\end{lem}
\begin{proof}
We use Theorem~\ref{thm:charactweakstarnull}. Given $t,s\in K$ with $t\ne s$, we have that $t$ and $s$ are both in $C_\epsilon$ only for finitely many $\epsilon\in\mathcal E$ and thus for any finite subset $S$ of $K$ we have that $\sum_{t\in S}\vert F_\epsilon(t)\vert\le1$ for all but finitely many $\epsilon\in\mathcal E$. To conclude the proof,
observe that $K^+$ is contained in $K^1$ and that for $t\in K^1$ we have $F_\epsilon(t)=0$ for all but finitely many $\epsilon\in\mathcal E$.
\end{proof}

Below it will be convenient to consider the topology $\tau$ on $K=L^\omega$ which is defined as the product topology with each factor $L$ endowed with the discrete topology. The topology $\tau$ is easier to work with than the order topology of $K$ and they are quite similar. More precisely, the topology $\tau$ is finer than the order topology of $K$ and for all $t\in K\setminus(K^0\cup K^1)$ the neighborhoods of $t$ in the order topology are the same as the neighborhoods of $t$ in the topology $\tau$.

Our strategy now is to find a suitable subset $Q$ of $K$ such that $(F_\epsilon)_{\epsilon\in\mathcal E}$ is not of type $c_0\ell_1$ over $Q$ and such that $Q\cap(H\setminus H^+)$ is countable, for every closed separable subset $H$ of $K$. We start by establishing a useful equivalence for the latter condition.
\begin{lem}\label{thm:lemmaEomega}
For a subset $Q$ of $K\setminus(K^0\cup K^1)$, the following conditions are equivalent:
\begin{itemize}
\item[(a)] $Q\cap E^\omega$ is countable for every countable subset $E$ of $L$;
\item[(b)] $Q\cap H$ is countable for every separable subset $H$ of $K$;
\item[(c)] $Q\cap(H\setminus H^+)$ is countable for every closed separable subset $H$ of $K$.
\end{itemize}
\end{lem}
\begin{proof}
Assume (a) and let us prove (b). If $H$ is a separable subset of $K$ then a countable dense subset of $H$ is contained in $E^\omega$ for some countable subset $E$ of $L$ and thus $H$ is contained in $\overline{E^\omega}$. Now, since $E^\omega$ is $\tau$-closed, we have that $\overline{E^\omega}$
is contained in $E^\omega\cup K^0\cup K^1$ and hence (a) and the fact that $Q\cap(K^0\cup K^1)=\emptyset$ imply that $Q\cap H$ is countable. Now assume (c) and let us prove (a). Given a countable subset $E$ of $L$ with $\max L\in E$ we show that $Q\cap E^\omega$ is countable. Since $E^\omega$ is $\tau$-separable, it follows that $E^\omega$ is separable and thus $H=\overline{E^\omega}$ is a closed separable subset of $K$. Using that $\max L\in E$ it is easily checked that
for all $t\in E^\omega\setminus K^1$ and every $s\in K$ with $s>t$ we have that $\left]t,s\right[\cap E^\omega\ne\emptyset$. Hence
$E^\omega\setminus K^1$ is disjoint from $H^+$ and it follows from (c) that $Q\cap(E^\omega\setminus K^1)=Q\cap E^\omega$ is countable.
\end{proof}

In order to write a helpful equivalence for the condition that $(F_\epsilon)_{\epsilon\in\mathcal E}$ is not of type $c_0\ell_1$ over a subset $Q$ of $K$,
we need to introduce a topological notion. We will say more about how such notion is related to families of type $c_0\ell_1$ in general later in Subsection~\ref{sub:countsepc0ell1}.
\begin{defin}\label{thm:defcountsep}
Given a topological space $X$, by a {\em countably separating family\/} for $X$ we mean a family $(V_x)_{x\in X}$ such that $V_x$ is a neighborhood of $x$, for all $x\in X$, and the following condition holds: given a subset $M$ of $X$, if $x\in V_y$ for all $x,y\in M$, then $M$ is countable.
\end{defin}

Below we will be interested also in certain subfamilies of $(F_\epsilon)_{\epsilon\in\mathcal E}$ rather than just in $(F_\epsilon)_{\epsilon\in\mathcal E}$ itself, so that later we will be able to prove that $K$ is not separably determined using a weak*-null family indexed by a set of cardinality $\omega_1$.
\begin{lem}\label{thm:lemmaseparating}
Let $R$ be a subset of $L\setminus\{\max L\}$ and $Q$ be a subset of $R^\omega$. The following conditions are equivalent:
\begin{itemize}
\item[(i)] the family $(F_\epsilon)_{\epsilon\in R^{<\omega}}$ is of type $c_0\ell_1$ over $Q$;
\item[(ii)] the set $Q$, endowed with the topology induced by $\tau$, admits a countably separating family.
\end{itemize}
\end{lem}
\begin{proof}
Assume (i) and write $F_\epsilon(t)=a_{\epsilon,t}+b_{\epsilon,t}$ for all $\epsilon\in R^{<\omega}$, $t\in Q$, with $\lim_{\epsilon\in R^{<\omega}}a_{\epsilon,t}=0$ for all $t\in Q$ and $\sum_{t\in Q}\vert b_{\epsilon,t}\vert\le k$ for all $\epsilon\in R^{<\omega}$ and some non negative real number $k$. For each $t\in Q$, the set
\begin{equation}\label{eq:epsilonsubsett}
\big\{\epsilon\in R^{<\omega}:\text{$\epsilon\subset t$ and $a_{\epsilon,t}\ge\tfrac12$}\big\}
\end{equation}
is finite and therefore there exists $\epsilon(t)\in R^{<\omega}$ such that $\epsilon(t)\subset t$ and $\epsilon(t)$ properly contains every element of
\eqref{eq:epsilonsubsett}. For all $t\in Q$, the set $V_t=C_{\epsilon(t)}\cap Q$ is an open neighborhood of $t$ in $Q$ with respect to the topology induced by $\tau$. In order to prove that $(V_t)_{t\in Q}$ is a countably separating family, we will show that there exists an upper bound on the size of finite
subsets $M$ of $Q$ such that $t\in V_s$ for all $t,s\in M$. Let $M$ be such a finite subset of $Q$. Since every $t\in M$ is a common extension of all
elements of $\{\epsilon(s):s\in M\}$, we have that $\epsilon=\bigcup_{s\in M}\epsilon(s)$ is in $R^{<\omega}$ and that every $t\in M$ extends $\epsilon$.
Note that, for $t\in M$, we have $t\in C_\epsilon$ and $t\ne\max C_\epsilon$ since $t\in R^\omega$ and
$\max L\not\in R$; thus $F_\epsilon(t)=1$. Moreover, since $\epsilon\subset t$ properly contains every element of \eqref{eq:epsilonsubsett}, we have that
$a_{\epsilon,t}<\frac12$ and therefore $b_{\epsilon,t}>\frac12$. Hence from $\sum_{t\in Q}\vert b_{\epsilon,t}\vert\le k$ we obtain that $M$ has at most $2k$ elements. This concludes the proof of (ii). Now assume (ii) and let us prove (i) by using the equivalence between (a) and (b) of Proposition~\ref{thm:propBicountable} (which is applicable to weak*-null families due to Proposition~\ref{thm:weakstarnullboundfinsums} and Lemma~\ref{thm:Fitnonzerocountable}). Let $(V_t)_{t\in Q}$ be a countably separating family for $Q$ endowed with the topology induced by $\tau$ and for each $t\in Q$ pick $\epsilon(t)\in R^{<\omega}$ with $\epsilon(t)\subset t$ and $C_{\epsilon(t)}\cap Q\subset V_t$. Set
\[A=\big\{(\epsilon,t)\in R^{<\omega}\times Q:\text{$F_\epsilon(t)=0$ or $\epsilon\subset\epsilon(t)$}\big\}\]
and $B=(R^{<\omega}\times Q)\setminus A$. For each $t\in Q$ we have both $(\epsilon,t)\in A$ and $F_\epsilon(t)\ne0$ only for finitely many $\epsilon\in R^{<\omega}$ and therefore $\lim_{\epsilon\in R^{<\omega}}F_\epsilon(t)\chilow A(\epsilon,t)=0$. To conclude the proof, we fix $\epsilon\in R^{<\omega}$ and we show that the set
\[B_\epsilon=\big\{t\in Q:(\epsilon,t)\in B\big\}\]
is countable. Since $(V_t)_{t\in Q}$ is countably separating, it is sufficient to check that $t\in V_s$ for all $t,s\in B_\epsilon$. Note that
for all $t\in B_\epsilon$ we have $F_\epsilon(t)\ne0$ and $\epsilon\not\subset\epsilon(t)$, which implies that $\epsilon(t)\subset\epsilon\subset t$.
Hence for all $t,s\in B_\epsilon$ we have $\epsilon(s)\subset\epsilon\subset t$, so that $t\in C_{\epsilon(s)}\cap Q\subset V_s$ and we are done.
\end{proof}

Now we turn to the construction of $Q$.
\begin{lem}\label{thm:lemmapressing}
Let the set $\omega_1^\omega$ be endowed with the product topology with each factor $\omega_1$ endowed with the discrete topology. There exists a subset $Q$
of $\omega_1^\omega$ that does not admit a countably separating family and such that $Q\cap\alpha^\omega$ is countable for every $\alpha\in\omega_1$.
\end{lem}
\begin{proof}
For each nonzero countable ordinal $\alpha$, let $t_\alpha:\omega\to\omega_1$ be a map whose image is $\alpha$ and set $Q=\big\{t_\alpha:0<\alpha<\omega_1\big\}\subset\omega_1^\omega$. For every $\alpha\in\omega_1$, we have that $Q\cap\alpha^\omega=\big\{t_\beta:0<\beta\le\alpha\big\}$ is countable. Now assume by contradiction that $Q$ admits a countably separating family $(V_t)_{t\in Q}$. For each $t\in\omega_1^\omega$ and $n\in\omega$, we denote by $U_n(t)$ the open neighborhood of $t$ in $\omega_1^\omega$ given by $U_n(t)=\big\{s\in\omega_1^\omega:s\vert_n=t\vert_n\big\}$ and for each $\alpha\in\left]0,\omega_1\right[$ we pick $n(\alpha)\in\omega$ such that $U_{n(\alpha)}(t_\alpha)\cap Q$ is contained in $V_{t_\alpha}$. We will now obtain a contradiction by showing that there exists an uncountable subset $M$ of $\left]0,\omega_1\right[$ such that $t_\alpha\in U_{n(\beta)}(t_\beta)$, for all $\alpha,\beta\in M$. Note that since a countable union of non stationary sets is non stationary, there exists a stationary subset $S$ of $\left]0,\omega_1\right[$ such that $n(\alpha)=m$ for all $\alpha\in S$ and some $m\in\omega$. Moreover, for each $i<m$, the map $\psi_i:S\to\omega_1$ given by $\psi_i(\alpha)=t_\alpha(i)$ satisfies $\psi_i(\alpha)<\alpha$ for all $\alpha\in S$ and thus applying the Pressing Down Lemma $m$ times we get a stationary subset $M$ of $S$ such that all of the maps $\psi_i$, $i<m$, are constant on $M$. Clearly,
such $M$ satisfies the desired condition.
\end{proof}

\begin{teo}\label{thm:notsepdet}
The compact line $K=L^\omega$ endowed with the lexicographic order is not separably determined, for any uncountable compact line $L$.
\end{teo}
\begin{proof}
Let $R$ be a subset of $L\setminus\{\min L,\max L\}$ with cardinality $\omega_1$. By Lemma~\ref{thm:lemmapressing} there exists a subset $Q$ of $R^\omega$ that does not admit a countably separating family in the topology induced by $\tau$ and such that $Q\cap E^\omega$ is countable for every countable subset $E$ of $R$.
By Lemma~\ref{thm:lemmaEomega}, $Q\cap(H\setminus H^+)$ is countable for every closed separable subset $H$ of $K$ and yet, by Lemma~\ref{thm:lemmaseparating}, the weak*-null family $(F_\epsilon)_{\epsilon\in R^{<\omega}}$ is not of type $c_0\ell_1$ over $Q$.
\end{proof}

The proof of Theorem~\ref{thm:notsepdet} yields an example of a continuous increasing surjection which has the $c_0$-extension property, but not the $c_0(\omega_1)$-extension property.
\begin{teo}\label{thm:c0EPbutnotc0omega1EP}
If $L$ is an arbitrary uncountable compact line and $K=L^\omega$ is endowed with the lexicographic order then there exists a compact line $\widetilde K$ and a continuous increasing surjection $\phi:\widetilde K\to K$ that has the $c_0$-extension property but does not have the $c_0(\omega_1)$-extension property.
\end{teo}
\begin{proof}
Pick $R$ and $Q$ as in the proof of Theorem~\ref{thm:notsepdet}. Set $\widetilde K=\DA(K;Q)$ and let $\phi:\widetilde K\to K$ denote the first projection, so that $\phi$ is a continuous increasing surjection and $Q(\phi)=Q$ (Remark~\ref{thm:remgenDA}). The weak*-null family $(F_\epsilon)_{\epsilon\in R^{<\omega}}$ is not of type $c_0\ell_1$ over $Q(\phi)$ and thus by Theorem~\ref{thm:c0ell1equivextend} it does not extend through $\phi$. Since the cardinality of $R^{<\omega}$ is $\omega_1$, we have that $\phi$ does not have the $c_0(\omega_1)$-extension property. On the other hand, by Lemma~\ref{thm:lemmaEomega}, $Q(\phi)\cap H=Q\cap H$ is countable for every separable subset $H$ of $K$ and hence by \cite[Theorem~2.6]{CTlines} we have that $\phi$ has the $c_0$-extension property.
\end{proof}

\begin{rem}\label{thm:remsmallercardinal}
Theorem~\ref{thm:c0EPbutnotc0omega1EP} gives an example of a closed subspace of a $C(K)$ space which has the $c_0$-extension property but not the $c_0(I)$-extension property
for a set $I$ of cardinality $\omega_1$. Another example of this kind can be obtained from a result in \cite{ArgyrosLondon}, though for a set $I$ of much larger cardinality.
More specifically, in \cite[Theorem~1.1~(b)]{ArgyrosLondon} the authors present an example of an Eberlein compact space $K$ and an isomorphic copy $Y$ of $c_0(I)$ in $C(K)$ which is not complemented. By \cite[Proposition~2.2~(b)]{CT}, $Y$ has the $c_0$-extension property in $C(K)$, yet $Y$ is not complemented and therefore does not have the $c_0(I)$-extension property in $C(K)$. The cardinality of $I$ in this example is quite large, namely $\beth_\omega$. Recall that $\beth_\omega=\sup_{n\in\omega}\beth_n$, where $\beth_0=\omega$ and $\beth_{n+1}=2^{\beth_n}$ for all $n\in\omega$.
\end{rem}

\subsection{Countably separating families and families of type $\mathbf{c_0\boldsymbol\ell_1}$}\label{sub:countsepc0ell1}

We show now that the relation that appeared in Lemma~\ref{thm:lemmaseparating} between spaces admitting countably separating families and weak*-null families of type $c_0\ell_1$ is connected to a deeper phenomenon. In this subsection $K$ denotes again an arbitrary compact line.

\begin{teo}\label{thm:queriasaberavolta}
Let $K$ be a compact line and $Q$ be a subset of $K$. If $Q$ endowed with the topology induced by $K$ admits a countably separating family then every weak*-null
family $(F_i)_{i\in I}$ in $\NBV(K)$ is of type $c_0\ell_1$ over $Q$.
\end{teo}

The following is an immediate consequence of Theorem~\ref{thm:queriasaberavolta}.
\begin{cor}\label{thm:corconjecture}
Let $K$ be a compact line, $Q$ be a subset of $K$ and $(F_i)_{i\in I}$ be a weak*-null family in $\NBV(K)$. If the set $\big\{t\in Q:\big(F_i(t)\big)_{i\in I}\not\in c_0(I)\big\}$ endowed with the topology induced by $K$ admits a countably separating family then $(F_i)_{i\in I}$ is of type $c_0\ell_1$ over $Q$.\qed
\end{cor}
We do not know if the converse of Corollary~\ref{thm:corconjecture} holds. A proof of such converse would yield a nice topological characterization of weak*-null families
of type $c_0\ell_1$ and hence (by Theorem~\ref{thm:c0ell1equivextend}) of $c_0(I)$-valued bounded operators defined on the subalgebra induced by a continuous increasing surjection $\phi:K\to L$ that admit a bounded extension to $C(K)$. This might also lead to interesting characterizations of continuous increasing surjections
with the $c_0(I)$-extension property, generalizing the main result of \cite{CTlines}.

\smallskip

For the proof of Theorem~\ref{thm:queriasaberavolta} we need two lemmas.

\begin{lem}\label{thm:FigeepsilonJinfinite}
Let $K$ be a compact line, $(F_i)_{i\in I}$ be a weak*-null family in $\NBV(K)$ and $J$ be a subset of $K$ with more than one point. If $J$ is either infinite or convex then for every $\varepsilon>0$ the set
\begin{equation}\label{eq:FitgeepsilonJinfinite}
\big\{i\in I:\text{$\vert F_i(t)\vert\ge\varepsilon$ for all $t\in J$}\big\}
\end{equation}
is finite.
\end{lem}
\begin{proof}
By Proposition~\ref{thm:weakstarnullboundfinsums}, setting $k=\sup_{i\in I}\Vert F_i\Vert_{\BV}$, we have
\begin{equation}\label{eq:Swithnelements}
\sum_{t\in S}\vert F_i(t)\vert\le\frac k2\mod c_0(I),
\end{equation}
for every finite subset $S$ of $K$. If $J$ is infinite, pick a subset $S$ of $J$ with $n>\frac k{2\varepsilon}$ elements and note that \eqref{eq:FitgeepsilonJinfinite}
is contained in $\big\{i\in I:\sum_{t\in S}\vert F_i(t)\vert\ge n\varepsilon\big\}$, which is a finite set by \eqref{eq:Swithnelements}. Now, if $J$ is convex and finite then $t=\min J$ is right-isolated and the finiteness of \eqref{eq:FitgeepsilonJinfinite} follows from $\lim_{i\in I}F_i(t)=0$.
\end{proof}

In the proof of next lemma we will use the fact that if $A$ is an arbitrary subset of a compact line $K$ then the maximal convex subsets of $K$ contained in $A$, known as the {\em convex components\/} of $A$ in $K$, constitute a partition of the set $A$.
\begin{lem}\label{thm:convexcomponents}
Let $K$ be a compact line and $F:K\to\R$ be a map of bounded variation. Given $\varepsilon_1>0$ and $\varepsilon_2>0$ with $\varepsilon_2<\varepsilon_1$, we have that the
set
\begin{equation}\label{eq:geepsilon1}
\big\{t\in K:F(t)\ge\varepsilon_1\big\}
\end{equation}
can be covered by finitely many disjoint convex subsets of $K$ contained in the set
\begin{equation}\label{eq:geepsilon2}
\big\{t\in K:F(t)\ge\varepsilon_2\big\}.
\end{equation}
\end{lem}
\begin{proof}
It is sufficient to show that the number of convex components of \eqref{eq:geepsilon2} in $K$ that intersect \eqref{eq:geepsilon1} is finite. To this aim, let
$(C_i)_{i=1}^n$ be distinct convex components of \eqref{eq:geepsilon2} that intersect \eqref{eq:geepsilon1} and let us find an upper bound for $n$.
For each $i=1,\ldots,n$, pick $t_i\in C_i$ with $F(t_i)\ge\varepsilon_1$; by reordering the indexes, we can assume that $t_1<t_2<\cdots<t_n$.
For each $i=1,\ldots,n-1$, as $t_i$ and $t_{i+1}$ belong to distinct convex components of \eqref{eq:geepsilon2}, we can find $s_i\in\left]t_i,t_{i+1}\right[$
with $F(s_i)<\varepsilon_2$. Now $\vert F(t_i)-F(s_i)\vert+\vert F(s_i)-F(t_{i+1})\vert>2(\varepsilon_1-\varepsilon_2)$ and thus setting
$P=\{t_1,s_1,t_2,s_2,\ldots,t_{n-1},s_{n-1},t_n\}$ we obtain
\[2(n-1)(\varepsilon_1-\varepsilon_2)\le V(F;P)\le V(F)\]
so that $n\le1+\frac{V(F)}{2(\varepsilon_1-\varepsilon_2)}$.
\end{proof}

\begin{proof}[Proof of Theorem~\ref{thm:queriasaberavolta}]
Let $(V_t\cap Q)_{t\in Q}$ be a countably separating family for $Q$, with $V_t$ a neighborhood of $t$ in $K$ for each $t\in Q$.
Our strategy is to construct subsets $A$ and $B$ of $I\times Q$ as in item~(b) of Proposition~\ref{thm:propBicountable} (which is applicable to weak*-null families due to Proposition~\ref{thm:weakstarnullboundfinsums} and Lemma~\ref{thm:Fitnonzerocountable}). For each positive integer $n$ and each $i\in I$, Lemma~\ref{thm:convexcomponents} applied to $F_i$ and to $-F_i$ yields finite collections $\mathcal J_i(n,{+})$ and $\mathcal J_i(n,{-})$ of pairwise disjoint convex subsets of $K$ such that:
\begin{gather}
\label{eq:calJinplus}\big\{t\in K:F_i(t)\ge\tfrac1n\big\}\subset\bigcup\mathcal J_i(n,{+})\subset\big\{t\in K:F_i(t)\ge\tfrac1{2n}\big\},\\
\big\{t\in K:F_i(t)\le-\tfrac1n\big\}\subset\bigcup\mathcal J_i(n,{-})\subset\big\{t\in K:F_i(t)\le-\tfrac1{2n}\big\}.
\end{gather}
Now for each positive integer $n$ we take $B(n,{+})$ (resp., $B(n,{-})$) to be the set of pairs $(i,t)\in I\times Q$ such that $t\in J$ and $J\subset V_t$ for some
$J\in\mathcal J_i(n,{+})$ (resp., for some $J\in\mathcal J_i(n,{-})$). We claim that for $i\in I$, the set
\[\big\{t\in Q:(i,t)\in B(n,{+})\big\}\]
is countable. Namely, such set is equal to the union of the sets
\begin{equation}\label{eq:tQtJVt}
\big\{t\in Q:t\in J\subset V_t\big\}
\end{equation}
with $J$ ranging over the finite set $\mathcal J_i(n,{+})$; moreover, for each $J\in\mathcal J_i(n,{+})$ the set \eqref{eq:tQtJVt} is countable since $t\in V_s\cap Q$ for all $t$ and $s$ in \eqref{eq:tQtJVt}. Similarly, one shows that the set
\[\big\{t\in Q:(i,t)\in B(n,{-})\big\}\]
is countable, for all $i\in I$. Thus, if we define
\[B=\bigcup_{n=1}^\infty\big(B(n,{+})\cup B(n,{-})\big)\]
we obtain that $B_i=\big\{t\in Q:(i,t)\in B\big\}$ is countable for all $i\in I$. To conclude the proof we now set $A=(I\times Q)\setminus B$ and we show that
$\lim_{i\in I}F_i(t)\chilow A(i,t)=0$ for all $t\in Q$. To this aim it is sufficient to fix $t\in Q$ and a positive integer $n$ and to prove that the sets
\begin{gather}
\label{eq:finitAnplus}\big\{i\in I:\text{$F_i(t)\ge\tfrac1n$ and $(i,t)\not\in B(n,{+})$}\big\},\\
\label{eq:finitAnminus}\big\{i\in I:\text{$F_i(t)\le-\tfrac1n$ and $(i,t)\not\in B(n,{-})$}\big\}
\end{gather}
are finite. In order to establish the finiteness of \eqref{eq:finitAnplus}, for each $i\in I$ with $F_i(t)\ge\frac1n$ pick $J_i\in\mathcal J_i(n,{+})$ such that
$t\in J_i$; note that if $(i,t)\not\in B(n,{+})$ then $J_i$ cannot be contained in $V_t$. Therefore the set \eqref{eq:finitAnplus} is contained in the union of the sets:
\begin{gather}
\label{eq:JitmaxK}\big\{i\in I:\text{$F_i(t)\ge\tfrac1n$ and $J_i$ intersects $\left]t,\max K\right]\setminus V_t$}\big\},\\
\label{eq:JiminKt}\big\{i\in I:\text{$F_i(t)\ge\tfrac1n$ and $J_i$ intersects $\left[\min K,t\right[\setminus V_t$}\big\}.
\end{gather}
If $t=\max K$, the set \eqref{eq:JitmaxK} is empty. Otherwise, there exists $s\in K$ with $s>t$ and $\left[t,s\right[\subset V_t$. Thus,
for every $i$ in \eqref{eq:JitmaxK}, the convex set $J_i$ intersects $[s,\max K]$ and so it contains $[t,s]$. From \eqref{eq:calJinplus} we now conclude that \eqref{eq:JitmaxK}
is contained in the set
\[\big\{i\in I:\text{$F_i(u)\ge\tfrac1{2n}$, for all $u\in [t,s]$}\big\}\]
which is finite by Lemma~\ref{thm:FigeepsilonJinfinite}. This proves that \eqref{eq:JitmaxK} is finite and a similar argument shows that \eqref{eq:JiminKt} is finite,
hence establishing that \eqref{eq:finitAnplus} is finite. Again, a similar argument shows that \eqref{eq:finitAnminus} is finite and we are done.
\end{proof}

\end{section}

\noindent\textbf{Acknowledgments.}\enspace The authors wish to thank Claudia Correa for valuable discussions during the preparation of this work. The first author was funded by CNPq (Conselho Nacional de Desenvolvimento Cient\'\i fico e Tecnol\'ogico, process number 141881/2017-8).


\begin{thebibliography}{99}


\bibitem{ArgyrosLondon} S. A. Argyros, J. F. Castillo, A. S. Granero, M. Jim\'enez \& J. P. Moreno,
Complementation and embeddings of $c_0(I)$ in Banach spaces, {\em Proc.\ London.\ Math.\ Soc.} 85 (3), 2002, pgs.\ 742---768.



\bibitem{Caudia} C. Correa, On the $c_0$-extension property, {\em Studia Math.} 256 (3), 2021, pgs.\ 345---359.

\bibitem{CT} C. Correa \& D. V. Tausk, On extensions of $c_0$-valued operators, {\em J. Math.\ Anal.\ Appl.\/} 405, 2013, pgs.\ 400---408.

\bibitem{CTJFA} C. Correa \& D. V. Tausk, Compact lines and the Sobczyk property, {\em J. Funct.\ Anal.} 266 (9), 2014, pgs.\ 5765---5778.

\bibitem{CTResults} C. Correa \& D. V. Tausk, Extension property and complementation of isometric copies of continuous functions spaces,
{\em Results in Math.} 67 (3--4), 2015, pgs.\ 445---455.

\bibitem{CTlines} C. Correa \& D.V. Tausk, On the $c_0$-extension property for compact lines, {\em J. Math.\ Anal.\ Appl.\/} 428 (1), 2015, pgs.\ 184---193.




%

\bibitem{DrPlebanek} P. Drygier \& G. Plebanek, Compactifications of $\omega$ and the Banach space $c_0$, {\em Fund.\ Math.} 237, 2017, pgs.\ 165---186.


\bibitem{Eloi} E. M. Galego \& A. Plichko, On Banach spaces containing complemented and uncomplemented subspaces
isomorphic to $c_0$, {\em Extracta Math.} 18, 2003, pgs.\ 315---319.

\bibitem{Hognas} G. H\"ogn\"as, Characterization of weak convergence of signed measures on $[0,1]$, {\em Math.\ Scand.} 41, 1977, pgs.\ 175---184.

\bibitem{KK} O. Kalenda \& W. Kubi\'s, Complementation in spaces of continuous functions on compact lines, {\em J. Math.\ Anal.\ Appl.\/} 386, 2012, pgs.\ 241---257.



%

\bibitem{Kubis} W. Kubi\'s, Linearly ordered compacta and Banach spaces with a projectional resolution of the identity,
{\em Topology Appl.\/} 154 (3), 2007, pgs.\ 749---757.


%


\bibitem{Molto} A. Molt\'o, On a theorem of Sobczyk, {\em Bull.\ Aust.\ Math.\ Soc.\/} 43 (1), 1991, pgs.\ 123---130.

\bibitem{Patterson} W. M. Patterson, Complemented $c_0$-subspaces of a non-separable $C(K)$-space,
{\em Canad.\ Math.\ Bull.\/} 36 (3), 1993, pgs.\ 351---357.


\bibitem{Phillips} R. S. Phillips, On linear transformations, {\em Trans.\ Amer.\ Math.\ Soc.} 48, 1940,
pgs.\ 516---541.




\bibitem{Sobczyk} A. Sobczyk, Projection of the space $(m)$ on its subspace $(c_0)$, {\em Bull.\ Amer.\
Math.\ Soc.} 47, 1941, pgs.\ 938---947.

%
%





\end{thebibliography}
\end{document}